\documentclass[a4paper,reqno]{amsart}

\title[Exotic Springer fibers for one-row bipartitions]{Exotic Springer fibers for orbits corresponding to one-row bipartitions}

\author{Neil Saunders}
\address{Department of Mathematical Sciences\\
	University of Greenwich\\
	United Kingdom; and \\
	Honorary Associate of the School of Mathematics and Statistics \\
	University of Sydney \\
	Australia}
\email{n.saunders@greenwich.ac.uk}

\author{Arik Wilbert}
\address{Department of Mathematics\\
	University of Georgia\\
	Athens, GA 30602\\
	USA}
\email{arik.wilbert@uga.edu}

\usepackage{geometry}
\usepackage[english]{babel}
\usepackage{amssymb}
\usepackage{multicol}
\usepackage{young}
\usepackage[vcentermath,enableskew]{youngtab}
\usepackage[shortlabels]{enumitem}
\usepackage{tikz}
\usetikzlibrary{arrows}
\usepackage{tikz-cd}
\usepackage[all, arrow, matrix, curve, knot, poly]{xy}
\usepackage{verbatim}
\usepackage{graphicx}

\theoremstyle{definition}
\newtheorem{defi}{Definition}
\theoremstyle{remark}
\newtheorem{ex}[defi]{Example}
\newtheorem{rem}[defi]{Remark}

\theoremstyle{plain}
\newtheorem{lem}[defi]{Lemma}

\newtheorem{cor}[defi]{Corollary}
\newtheorem{prop}[defi]{Proposition}
\newtheorem{thm}[defi]{Theorem}

\newtheorem{introthm}{Theorem}

\newcommand{\CupConnect}{\text{\----}}
\newcommand{\RayConnect}{\mid\hspace{-5pt}\text{\----}\,i}
\newcommand{\ba}{{\bf a}}
\newcommand{\bb}{{\bf b}}
\newcommand{\bigslant}[2]{{\raisebox{.2em}{$#1$}\left/\raisebox{-.2em}{$#2$}\right.}}

\DeclareMathOperator{\spn}{span}

\begin{document}

\begin{abstract}
We study the geometry and topology of exotic Springer fibers for orbits corresponding to one-row bipartitions from an explicit, combinatorial point of view. This includes a detailed analysis of the structure of the irreducible components and their intersections as well as the construction of an explicit affine paving. Moreover, we compute the ring structure of cohomology by constructing a CW-complex homotopy equivalent to the exotic Springer fiber. This homotopy equivalent space admits an action of the type C Weyl group inducing Kato's original exotic Springer representation on cohomology. Our results are described in terms of the diagrammatics of the one-boundary Temperley--Lieb algebra (also known as the blob algebra). This provides a first step in generalizing the geometric versions of Khovanov's arc algebra to the exotic setting. 
\end{abstract}

\maketitle

\section{Introduction}\label{sec:intro}

In~\cite{Kat09}, Kato introduced the exotic nilpotent cone in a successful attempt to extend the Kazhdan--Lusztig--Ginzburg geometrization of affine Hecke algebras (see \cite{KL87}, \cite{CG97}) from the one-parameter to the multi-parameter case by considering the equivariant algebraic $K$-theory of an exotic version of the Steinberg variety. Kato's construction establishes a Deligne--Langlands type classification of irreducible modules of affine Hecke algebras of type C with only very mild restrictions on the parameters. Let $\mathcal{N}(\mathfrak{gl}_{2m})\subseteq\mathfrak{gl}_{2m}(\mathbb C)$ be the ordinary nilpotent cone of type A and let $\mathcal S\subseteq\mathfrak{gl}_{2m}(\mathbb C)$ denote the $\mathrm{Sp}_{2m}(\mathbb C)$-invariant complement of $\mathfrak{sp}_{2m}(\mathbb C)$ in $\mathfrak{gl}_{2m}(\mathbb C)$. The exotic nilpotent cone is the singular affine variety $\mathfrak{N}=\mathbb C^{2m}\times\left(\mathcal S \cap \mathcal{N}(\mathfrak{gl}_{2m})\right)$. Many of its properties relating to, e.g., intersection cohomology of orbit closures (see~\cite{AH08} and~\cite{SS14}), theory of special pieces (see~\cite{AHS11}), and the Lusztig--Vogan bijection (see~\cite{Nan13}) have been explored in follow-up work to~\cite{Kat09}.

In~\cite{Kat09}, \cite{Kat11}, Kato used the exotic nilpotent cone to construct an exotic version of the Springer correspondence for the Weyl group of type C. This correspondence is less intricate than the classical type C Springer correspondence (see~\cite{Spr76},~\cite{Spr78},~\cite{Sho79}) because it gives in fact a bijection between the orbits under the $\mathrm{Sp}_{2m}(\mathbb C)$-action on $\mathfrak{N}$ and the isomorphism classes of all complex, finite-dimensional, irreducible representations of the type C Weyl group (i.e., it behaves more like the classical type A Springer correspondence for the symmetric group). In particular, since the irreducible representations of the Weyl group can be labeled by bipartitions (see~\cite[Appendix B]{Mac95}), it follows that the orbits in $\mathfrak N$ can be labeled by bipartitions (see~\cite{AH08} for an explicit bijection between exotic nilpotent orbits and bipartitions). 

In analogy to classical Springer theory, the irreducible Weyl group representations can be constructed on the top non-vanishing cohomology of so-called exotic Springer fibers. These algebraic varieties appear as the fibers under a resolution of singularities of the exotic nilpotent cone. More explicitly, given $(v,x)\in\mathfrak{N}$, the exotic Springer fiber $\mathcal Fl_{(v,x)}$ is the projective variety consisting of all flags $\{0\}=F_0\subsetneq F_1\subsetneq F_2\subsetneq\ldots\subsetneq F_m\subsetneq \mathbb C^{2m}$ such that $F_m$ is a Lagrangian subspace in $\mathbb C^{2m}$ (with respect to a fixed symplectic form), $v\in F_m$ and $xF_i\subseteq F_{i-1}$, $1\leq i\leq m$. The varieties $\mathcal Fl_{(v,x)}$ depend (up to isomorphism) only on the $\mathrm{Sp}_{2m}(\mathbb C)$-orbit of $(v,x)$ in $\mathfrak{N}$ which justifies the notation $\mathcal Fl^{(\lambda,\mu)}$, where $(\lambda,\mu)$ is the bipartition of $m$ labeling the orbit of $(v,x)$. The geometric structure of the exotic Springer fibers is only poorly understood. In general, they are not smooth and decompose into many irreducible components. Note that even for the classical Springer fibers of type A, the irreducible components themselves are not smooth in general (see \cite{FM10} for a parametrization of all smooth irreducible components in type A in terms of standard tableaux).

In this article we study exotic Springer fibers for orbits corresponding (via the bijection from~\cite{AH08}) to one-row bipartitions of $m$, i.e., pairs of partitions of the form $((k),(m-k))$, $0\leq k\leq m$, where each partition consists of a single number only. In this case we show that the irreducible components exhibit structural properties very similar to the ones discovered by Fung, \cite{Fun03}, for two-row Springer fibers of type A (see also~\cite{ES12} for type D). The first main result proved in Section~\ref{sec:structure_irred_comp} (see also Theorem~\ref{thm:algebraic_main_result}\ref{third_part_thm}) is the following:

\begin{introthm}
Every irreducible component of $\mathcal Fl^{((k),(m-k))}$ is an $(m-k)$-fold iterated fiber bundle over $\mathbb P^1$. In particular, all irreducible components of $\mathcal Fl^{((k),(m-k))}$ are smooth.
\end{introthm}

A combinatorial parametrization of the irreducible components of exotic Springer fibers in the spirit of \cite{Spa76}, \cite{Var79}, for type A (see also \cite{Spa82}, \cite{vL89}, for the remaining classical types) was only recently obtained in~\cite{NRS16} using standard bitableaux. Our results are based on a new parametrization in terms of so-called one-boundary cup diagrams. These diagrams enable us to write down explicit relations between the vector spaces of the flags contained in a given irreducible component (see Theorem~\ref{thm:algebraic_main_result}\ref{second_part_thm} for details). This is similar to the results obtained by Stroppel--Webster~\cite[Section 1.2]{SW12} for type A two-row Springer fibers based on earlier work by Fung~\cite{Fun03} (see also~\cite{ILW19} for the remaining classical types).

Consider a rectangle in the plane with $m$ vertices evenly spread along its upper horizontal edge. A one-boundary cup diagram is a diagram inside the rectangle obtained by either connecting two vertices by a lower semicircle called a cup, or by connecting a vertex with a point on the lower horizontal edge by a vertical line segment called a ray, or by connecting a vertex with a point on the right vertical edge of the rectangle by a half-cup. In doing so, we require that the diagram is crossingless and every vertex is connected to exactly one endpoint of either a cup, a half-cup or a ray. If the cups, half-cups and rays of two given cup diagrams are incident with exactly the same vertices in both diagrams (regardless of the actual drawing of the cups, half-cups and rays in the plane) we consider the diagrams as equal. Here is an example:
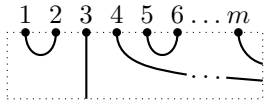
\[
\begin{tikzpicture}[scale=.8,baseline={(0,-.45)}]
\draw[dotted] (-.3,-1.1) -- (3.9,-1.1) -- (3.9,0) -- (-.3,0) -- cycle;
\node[above] at (0,0) {$1$};
\node[above] at (0.5,0) {$2$};
\node[above] at (1,0) {$3$};
\node[above] at (1.5,0) {$4$};
\node[above] at (2,0) {$5$};
\node[above] at (2.5,0) {$6$};
\node[above] at (3,0) {$\dots$};
\node[above] at (3.5,0) {$m$};

\foreach \x in {0,0.5,1,...,2.5}
\fill (\x,0) circle (2pt);
\fill (3.5,0) circle (2pt);

\draw[thick] (0,0) .. controls +(0,-.5) and +(0,-.5) .. +(.5,0);
\draw[thick] (2,0) .. controls +(0,-.5) and +(0,-.5) .. +(.5,0);

\draw[thick] (1,0) -- (1,-1.1);
\draw[thick] (1.5,0) to[out=270,in=170] (2.65,-0.7);
\node [rotate=357] at (3,-0.75) {$\cdots$};
\draw[thick] (3.3,-.75) -- (3.9,-.8);

\draw[thick, rotate around={180:(2,0)}] (0.5,0) arc (0:75:5.4mm); 
\end{tikzpicture}
\] 

In Section~\ref{sec:torus_fixed_points} we show that the set of all one-boundary cup diagrams on $m$ vertices with a total number of at most $m-k$ cups and half-cups combinatorially describes a Bia\l{}ynicki-Birula paving of $\mathcal Fl^{((k),(m-k))}$. More precisely, each one-boundary cup diagram corresponds to an affine cell and the respective diagram can be used to explicitly determine all flags contained in a given cell (see Theorem~\ref{thm:bialynicki-birula}). The closures of the cells corresponding to the diagrams with the number of cups and half-cups being maximal are exactly the irreducible components of $\mathcal Fl^{((k),(m-k))}$. In particular, we get a bijection between the set $\mathbb B^{((k),(m-k))}$ of all one-boundary cup diagrams on $m$ vertices with a total number of $m-k$ cups and half-cups and the irreducible components of $\mathcal Fl^{((k),(m-k))}$ (see also Theorem~\ref{thm:algebraic_main_result}\ref{first_part_thm}).

One-boundary cup diagrams naturally appear in the representation theory of the one-boundary Temperley--Lieb algebra. They label a basis of the so-called link pattern representation which appears in the study of the XXZ quantum chain,~\cite{NRG05}. In the mathematics literature, the one-boundary Temperley--Lieb algebra is more commonly known as the blob algebra,~\cite{MS94}. Instead of working with the one-boundary cup diagrams, one typically uses certain decorated cup diagrams which do not contain half-cups. We refer the reader to Remark~\ref{rem:diagrammatics_dictionary} for details and a concrete dictionary between the diagrammatics. We remark that the one-boundary Temperley--Lieb algebra or blob algebra is naturally a quotient of the two-parameter Hecke algebra for the Weyl group of type C,~\cite{MW03}. This explains the appearance of the one-boundary cup diagrams when studying exotic Springer fibers.

In Section~\ref{sec:Topological_Springer_fiber} we use the one-boundary cup diagrams to construct a topological model of the exotic Springer fiber $\mathcal Fl^{((k),(m-k))}$ similar to the topological models of two-row Springer fibers constructed in \cite{Kho04}, \cite{Weh09}, \cite{RT11}, \cite{Rus11}, for type A and in \cite{ES12}, \cite{Wil18}, for the remaining classical two-row types. Let $\mathbb S^2 \subseteq \mathbb R^3$ be the two-dimensional standard unit sphere with north pole $p=(0,0,1)$. Given a cup diagram $\ba\in\mathbb B^{((k),(m-k))}$, define $S_\ba\subseteq\left(\mathbb S^2\right)^m$ as the submanifold consisting of all $(x_1,\ldots,x_m)\in\left(\mathbb S^2\right)^m$ satisfying the relations $x_j=-x_i$ if vertices $i$ and $j$ are connected by a cup in $\ba$, and $x_i=p$ if vertex $i$ is connected to a ray in $\ba$. There are no relations involving coordinates indexed by a vertex connected to a half-cup. This yields a topological model for the exotic Springer fiber, i.e., we obtain the following homeomorphism (see also Theorem~\ref{thm:main_result_1}):

\begin{introthm}\label{introthm:top_model}
There exists a homeomorphism 
\[
\bigcup_{\ba \in \mathbb B^{((k),(m-k))}}S_\ba\cong\mathcal Fl^{((k),(m-k))}
\]
such that the images of the $S_\ba$ are the irreducible components of $\mathcal Fl^{((k),(m-k))}$.
\end{introthm}

We equip $\mathbb S^2$ with the structure of a CW-complex consisting of a $0$-cell and a $2$-cell and equip $(\mathbb S^2)^m$ with the product CW-structure. Let $\mathrm{Sk}_{2(m-k)}^m$ denote the $2(m-k)$-skeleton of $(\mathbb S^2)^m$. The next theorem explicitly identifies the homotopy type of the exotic Springer fiber $\mathcal Fl^{((k),(m-k))}$, i.e., we show that $\mathrm{Sk}_{2(m-k)}^m$ is isomorphic (but in general not homeomorphic, see Example~\ref{ex:homeo_vs_homotopy}) to $\mathcal Fl^{((k),(m-k))}$ in the homotopy category of topological spaces.

\begin{introthm}\label{introthm:cohomology}
There exists a homotopy equivalence $\mathcal Fl^{((k),(m-k))}\simeq\mathrm{Sk}_{2(m-k)}^m$. In particular, the cohomology ring $H^*(\mathcal Fl^{((k),(m-k))},\mathbb C)$ is isomorphic (as a graded algebra) to $H^*(\mathrm{Sk}_{2(m-k)}^m,\mathbb C)$ which has an explicit presentation given by 
\[
\bigslant{\mathbb C[X_1,\ldots,X_m]}{\left\langle X_i^2, X_I\;
  \begin{array}{|c}
  1 \leq i \leq m,\\
  I\subseteq  \{1,\ldots,m\}, |I|=m-k+1
  \end{array}
  \right\rangle},
\]
where $X_I=\prod_{i\in I}X_i$ and $\mathrm{deg}(X_i)=2$.
\end{introthm}

Even in the case of one-row bipartitions, the type C Weyl group does not act on the exotic Springer fiber (and neither on its topological model from Theorem~\ref{introthm:top_model}). However, it acts on the homotopy equivalent space $\mathrm{Sk}_{2(m-k)}^m$ in a natural way and induces an action on cohomology which can be described explicitly (see Proposition~\ref{prop:Weyl_group_action_cohomology}). We prove that the resulting representation $H^{2l}(\mathrm{Sk}_{2(m-k)}^m,\mathbb C)$ is isomorphic to the irreducible Specht module of the type C Weyl group labeled by $((m-l),(l))$, $0\leq l \leq m-k$. In particular, the top degree representation is as expected from the exotic Springer correspondence~\cite{Kat11}. In fact, by using results from~\cite{Kat17}, it follows that our representations coincide with the ones constructed by Kato in all cohomological degrees. Furthermore, our explicit presentation of $H^*(\mathrm{Sk}_{2(m-k)}^m,\mathbb C)$ comes with a distinguished basis of $H^{2l}(\mathrm{Sk}_{2(m-k)}^m,\mathbb C)$ given by monomials with $l$ pairwise different factors. We show that this monomial basis corresponds to the well known bipolytabloid basis indexed by standard bitableaux of the Specht module, see~\cite{Can96}. It would be interesting to extend the construction of homotopy types which admit an action of the Weyl group to more general Springer fibers.

Finally, we describe the structure of the pairwise intersections of the irreducible components of the exotic Springer fiber $\mathcal Fl^{((k),(m-k))}$ using circle diagrams (see Theorem~\ref{thm:intersection_of_components}). A circle diagram is obtained by putting the cup diagram assigned to a component upside down on top of the diagram associated to another component. We can then consider the vector space $\mathbb K_m$ whose basis is given by all such oriented circle diagrams and obtain the following result. 

\begin{introthm} 
There is an isomorphism of vector spaces 
\begin{equation}\label{eq:convolution_algebra_iso}
\mathbb K_m \cong \bigoplus_{(\ba,\bb)\in\left(\mathbb B^{((k),(m-k))}\right)^2} H^*\left(K_\ba\cap K_\bb,\mathbb C\right),
\end{equation}
where $K_\ba$ and $K_\bb$ are the irreducible components of $\mathcal Fl^{((k),(m-k))}$ labeled by the cup diagrams $\ba,\bb\in\mathbb B^{((k),(m-k))}$, respectively.
\end{introthm}

The right-hand side of~(\ref{eq:convolution_algebra_iso}) can be equipped with a convolution product by mimicking the construction in~\cite{SW12}. This product should have a diagrammatic description in terms of the one-boundary cup diagram combinatorics on the left-hand side of~(\ref{eq:convolution_algebra_iso}). For two-row Springer fibers of type A such a diagrammatic description was obtained in~\cite{SW12}. This gives a geometric construction of the (generalized) arc algebras of type A (including their quasi-hereditary covers) which were combinatorially defined in~\cite{Str09},\cite{BS11} (see also~\cite{CK14}) based on \cite{Kho02}. These algebras do not only appear in low-dimensional topology in the context of defining tangle homology theories, but they are known to be related to many interesting representation theoretic categories such as parabolic category $\mathcal O$,~\cite{BS11categoryO}, perverse sheaves on Grassmannians,~\cite{Str09}, non-semisimple representation theory of the Brauer algebra,~\cite{BS12brauer}, and the general linear Lie superalgebra,~\cite{BS12}. Our article can be seen as providing a geometric framework for constructing similar convolution algebras for exotic Springer fibers using one-boundary cup diagrams. However, establishing relationships between interesting representation theoretic categories comparable to the aforementioned results in type A is expected to be nontrivial in the exotic case. In particular, since the resolution of the exotic nilpotent cone is not symplectic, but only Calabi--Yau, the involved categories would necessarily have very different properties from the ones arising from the type A convolution algebras. Nonetheless, based on, e.g.,~\cite{tD94},~\cite{OR07}, we expect that the exotic arc algebras should play a role in defining homological invariants of tangles in a thickened annulus which should be compared to the constructions in~\cite{APS04},~\cite{GLW18} and~\cite{RT19}. Moreover, they should play a role in categorifying certain induced modules for multiparameter Hecke algebras of type C (at least for a very specific choice of parameters). Furthermore, it would be interesting to see how our geometry relates to the Schur--Weyl dualities recently studied in, e.g.,~\cite{BWW18} or~\cite{LV20}. 

\subsection*{Acknowledgments}

The authors would like to thank Nora Ganter, Anthony Henderson, Vinoth Nandakumar, Arun Ram and Catharina Stroppel for useful discussions and comments. We are particularly grateful to Maud De Visscher for inviting A.W.\ to visit City, University of London, where this project began.

Parts of this research were completed during the Junior Hausdorff Trimester Program ``Symplectic Geometry and Representation Theory''  at the Hausdorff Research Insitute (HIM) in Bonn. A.W.\ would like to thank the HIM for providing an excellent research environment. A.W.\ was also partially supported by the ARC Discovery Grant DP160104912 ``Subtle Symmetries and the Refined Monster''.

The authors would like to thank two anonymous referees for very detailed and helpful comments.

\section{Exotic Springer fibers}\label{sec:basic_defs}

We begin by recalling some basic definitions and facts concerning exotic Springer fibers and their irreducible components. Moreover, we define the one-boundary cup diagrams which are the central diagrammatic tool in this article. 

\subsection{Exotic nilpotent cone and exotic Springer fibers}

Throughout this article, we fix a positive integer $m>0$. Let $V$ be a $2m$-dimensional complex vector space with basis $e_1,\ldots,e_m,f_1,\ldots,f_m$ and symplectic form $\omega$ given by $\omega(e_i,f_j)=\delta_{i+j,m+1}=-\omega(f_j,e_i)$ and $\omega(e_i,e_j)=\omega(f_i,f_j)=0$. Let $\mathrm{Sp}(V,\omega)$ be the symplectic group of all linear automorphisms of $V$ preserving $\omega$ and let $\mathfrak{sp}(V,\omega)$ be its Lie algebra, i.e., the Lie subalgebra of the general linear Lie algebra $\mathfrak{gl}(V)$ consisting of all linear endomorphisms $x$ of $V$ such that $\omega(xv,w)=-\omega(v,xw)$ for all $v,w\in V$. The adjoint action of $\mathrm{Sp}(V,\omega)$ on $\mathfrak{gl}(V)$ yields a direct sum decomposition $\mathfrak{gl}(V)=\mathfrak{sp}(V,\omega)\oplus\mathcal S(V,\omega)$ of $\mathrm{Sp}(V,\omega)$-modules, where
\[
\mathcal S(V,\omega)=\big\{x\in\mathfrak{gl}(V)\mid \omega(xv,w)-\omega(v,xw)=0,\,\,\forall v,w\in V\big\}.
\]

\begin{defi}
The {\it exotic nilpotent cone} is the affine variety defined as 
\[
\mathfrak{N}=V\times\left(\mathcal S(V,\omega)\cap\mathcal N(\mathfrak{gl}(V))\right),\] 
where $\mathcal N(\mathfrak{gl}(V))\subseteq\mathfrak{gl}(V)$ is the ordinary nilpotent cone consisting of all nilpotent (in the usual sense of linear algebra) linear endomorphisms of $V$.
\end{defi}

Given a partition $\lambda$, we write $\vert\lambda\vert$ to denote the sum $\displaystyle \sum_{i \geq 1} \lambda_i$ of its parts and $\ell(\lambda)$ denotes the number of parts. Moreover, $\lambda+\mu$ is the partition defined by $(\lambda+\mu)_i=\lambda_i+\mu_i$ and $\lambda\cup\mu$ is the partition obtained by taking both the parts of $\lambda$ and $\mu$ and order them such that $\lambda\cup\mu$ is a partition. A {\it bipartition} of $m$ is a pair $(\lambda,\mu)$ of partitions such that $\vert\lambda\vert+\vert\mu\vert=m$. 

\begin{prop}[{\cite[Theorem 6.1]{AH08}}] \label{prop:classification_nilpotent_orbits}
The orbits under the $\mathrm{Sp}(V,\omega)$-action on $\mathfrak{N}$ are in bijective correspondence with bipartitions of $m$. 
\end{prop}

\begin{rem} \label{rem:explicit_AH_bijection}
It follows from Section $2$ and Section $6$ of \cite{AH08} (see also \cite[Theorem 2.10]{NRS17}) that a point $(v,x)\in\mathfrak{N}$ is contained in the orbit labeled by the bipartition $(\lambda,\mu)$ of $m$ if and only if there is a basis of $V$ given by
\[
\big\{e_{ij},f_{ij}\mid 1\leq i\leq \ell(\lambda+\mu),1\leq j\leq\lambda_i+\mu_i\big\},
\] 
where $\omega(e_{ij},f_{i'j'})=\delta_{i,i'}\delta_{j+j',\lambda_i+\mu_i+1}$, $v=\displaystyle \sum_{i=1}^{\ell(\lambda)}e_{i,\lambda_i}$ and such that the action of $x$ on this basis is as follows:
\[
xe_{ij}=\begin{cases}
e_{i,j-1} &\text{if }j\geq 2\\
0					&\text{if }j=1
\end{cases}
 \hspace{1.6em} xf_{ij}=\begin{cases}
f_{i,j-1} &\text{if }j\geq 2\\
0					&\text{if }j=1.
\end{cases}
\]
Note that $x$ has Jordan type $(\lambda+\mu)\cup(\lambda+\mu)$. 
\end{rem}

By Remark~\ref{rem:explicit_AH_bijection} we obtain the following distinguished representatives of orbits corresponding to one-row bipartitions.

\begin{lem}
Let $x\in\mathcal{N}(\mathfrak{gl}(V))$ be the nilpotent endomorphism of $V$ whose action on $V$ is given as follows:
\begin{equation}\label{eq:Jordan_basis_vectors}
\begin{tikzpicture}[baseline={(0,0)}]
\node (e1) at (0,0) {$e_1$};
\node (e2) at (1,0) {$e_2$};
\node (space1) at (2,0) {$\ldots{}^{}$};
\node (eN) at (3.1,0) {$e_m$};
\node (f1) at (5,0) {$f_1$};
\node (f2) at (6,0) {$f_2$};
\node (space2) at (7,0) {$\ldots{}^{}$};
\node (fN) at (8.1,0) {$f_m$.};
\path[->,font=\scriptsize,>=angle 90,bend right]
(e2) edge (e1)
(space1) edge (e2)
(eN) edge (space1)
(f2) edge (f1)
(space2) edge (f2)
(fN) edge (space2);
\end{tikzpicture}
\end{equation}
The vectors $e_1$ and $f_1$ are sent to $0$. Then the point $(e_k,x)\in\mathfrak{N}$, where $e_k\in V$ is the $k$th Jordan basis vector of the first block in (\ref{eq:Jordan_basis_vectors}), is contained in the orbit labeled by the bipartition $((k),(m-k))$.
\end{lem}

Let $\mathcal Fl_C$ be the flag variety of type $C$ consisting of all flags $\{0\}\subsetneq F_1\subsetneq F_2\subsetneq\ldots\subsetneq F_m\subsetneq V$ such that $\dim_{\mathbb C}(F_i)=i$ and $F_i\subseteq V$ is isotropic with respect to $\omega$, $1\leq i\leq m$. The exotic nilpotent cone $\mathfrak{N}$ has a resolution of singularities $\pi\colon\widetilde{\mathfrak{N}}\to\mathfrak{N}$, where
\[
\widetilde{\mathfrak{N}} = \big\{\left((v,x),(F_1,\ldots,F_m)\right)\in\mathfrak{N}\times\mathcal Fl_C\mid xF_i\subseteq F_{i-1}\,,v\in F_m\big\}
\]
and $\pi$ is the projection onto the first component,~\cite{Kat11}. Analogous to the ordinary Springer fibers, we can now define the exotic Springer fibers as the fibers of elements $(v,x)\in\mathfrak{N}$ under the resolution $\pi$. 

\begin{defi} \label{defi_exotic_springer_fiber}
Given $(v,x)\in\mathfrak{N}$, the {\it exotic Springer fiber} $\mathcal Fl_{(v,x)}$ is the variety consisting of all flags $\mathcal Fl_C$ satisfying $v\in F_m$ and $xF_i\subseteq F_{i-1}$. 
\end{defi}

Up to isomorphism of algebraic varieties, the exotic Springer fiber only depends on the $\mathrm{Sp}(V,\omega)$-orbit of $(v,x)$ in $\mathfrak{N}$ and not on the chosen point itself. This allows us to talk about the $(\lambda,\mu)$-Springer fiber denoted by $\mathcal Fl^{(\lambda,\mu)}$ without ambiguity.

\subsection{Classification of irreducible components} \label{sec:irred_comp}

Let $(\lambda,\mu)$ be a bipartition of $m$. A bitableau of shape $(\lambda,\mu)$ is a filling of the $m$ boxes of the underlying pair of Young diagrams with the numbers $1,2,\ldots,m$ such that each number appears exactly once. A bitableau is called standard if the entries increase (to the right and downwards) along rows and columns. Note that this convention differs from the one used in~\cite{NRS16},~\cite{NRS17}. 

\begin{prop}[{\cite[Theorem 2.12]{NRS16}}] \label{prop:parametrization_of_components}
There is a bijection between the irreducible components of the $(\lambda,\mu)$ exotic Springer fiber and the set of all standard bitableaux of shape $(\lambda,\mu)$. Moreover, each irreducible component has the same dimension, which is $2M(\lambda+\mu)+\vert\mu\vert$. Here, $M(\lambda+\mu)=\displaystyle \sum_{i \geq 1} (i-1)(\lambda_i + \mu_i)$ is the dimension of the Springer fiber of type $A$ corresponding to $\lambda+\mu$. 
\end{prop}

\subsection{One-boundary cup diagrams}

We finish this section by defining the one-boundary cup diagrams. They can be used as yet another combinatorial means to parametrize the irreducible components of the exotic Springer fibers in the case of one-row bipartitions (see Lemma~\ref{lem:bijection_tableaux_cups} and Remark~\ref{rem:dimension} below). The advantage of this parametrization is that it captures the topology and geometric structure of the irreducible components in a neat and very explicit way (see Subsection~\ref{subsec:irred_components} for details). For the remainder of this article we fix a one-row bipartition $((k),(m-k))$ of the positive integer $m$, $0\leq k\leq m$.

\begin{defi}\label{defi:one-boundary_cup_diagrams}
Consider a rectangle in the plane with $m$ vertices evenly spread along the upper horizontal edge. The vertices are labeled by the consecutive integers $1,\ldots,m$ in increasing order from left to right.

A {\it one-boundary cup diagram} (or shorter {\it cup diagram}) is obtained by either connecting two vertices by a lower semicircle called a {\it cup}, or by connecting a vertex with a point on the lower horizontal edge by a vertical line segment called a {\it ray}, or by connecting a vertex with a point on the right vertical edge of the rectangle by a {\it half-cup}. In doing so we require that the diagram is crossingless and every vertex is an endpoint of exactly one cup, ray or half-cup. If the cups, rays and half-cups of two given cup diagrams are incident with exactly the same vertices in both diagrams (regardless of the precise drawing of the cups, rays and half-cups) we consider the diagrams as equal. The notation $i\CupConnect j$ means that $i<j$ are connected by a cup and $\RayConnect$ means that $i$ is connected to a ray. For our purposes, we do not need a short notation for a vertex connected to a half-cup. 

We write $\mathbb B^{((k),(m-k))}$ to denote the set of all cup diagrams on $m$ vertices such that the number of cups plus the number of half-cups equals $m-k$.
\end{defi}

\begin{ex} \label{ex:cup_diagrams}
The set $\mathbb B^{((3),(1))}$ consists of the cup diagrams
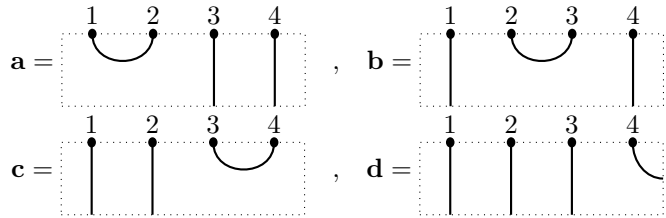
\[
\begin{array}{cccc}
\ba = 
\begin{tikzpicture}[scale=0.8,xscale=1,yscale=1.2,baseline={(0,0.5)}]
\draw[dotted] (0.5,0) -- (0.5,1) -- (4.5,1) -- (4.5,0) -- cycle;
\node[above] at (1,1) {$1$};
\node[above] at (2,1) {$2$};
\node[above] at (3,1) {$3$};
\node[above] at (4,1) {$4$};
\foreach \x in {1,...,4}{
\fill (\x,1) circle (2pt);}

\draw[thick] (1,1) .. controls +(0,-0.5) and +(0,-0.5) .. (2,1);
\draw[thick] (3,1) -- +(0,-1);
\draw[thick] (4,1) -- +(0,-1);
\end{tikzpicture}
&,&
\bb = \begin{tikzpicture}[scale=0.8,xscale=1,yscale=1.2,baseline={(0,0.5)}]
\draw[dotted] (0.5,0) -- (0.5,1) -- (4.5,1) -- (4.5,0) -- cycle;
\node[above] at (1,1) {$1$};
\node[above] at (2,1) {$2$};
\node[above] at (3,1) {$3$};
\node[above] at (4,1) {$4$};
\foreach \x in {1,...,4}{
\fill (\x,1) circle (2pt);}

\draw[thick] (1,1) -- +(0,-1);
\draw[thick] (2,1) .. controls +(0,-0.5) and +(0,-0.5) .. (3,1);
\draw[thick] (4,1) -- +(0,-1);

\end{tikzpicture}
\\
{\bf c} =  \begin{tikzpicture}[scale=0.8,xscale=1,yscale=1.2,baseline={(0,0.5)}]
\draw[dotted] (0.5,0) -- (0.5,1) -- (4.5,1) -- (4.5,0) -- cycle;
\node[above] at (1,1) {$1$};
\node[above] at (2,1) {$2$};
\node[above] at (3,1) {$3$};
\node[above] at (4,1) {$4$};

\foreach \x in {1,...,4}{
\fill (\x,1) circle (2pt);}

\draw[thick] (1,1) -- +(0,-1);
\draw[thick] (3,1) .. controls +(0,-0.5) and +(0,-0.5) .. (4,1);
\draw[thick] (2,1) -- +(0,-1);
\end{tikzpicture}
&,&
{\bf d} =   \begin{tikzpicture}[scale=0.8,xscale=1,yscale=1.2,baseline={(0,0.5)}]
\draw[dotted] (0.5,0) -- (0.5,1) -- (4.5,1) -- (4.5,0) -- cycle;
\node[above] at (1,1) {$1$};
\node[above] at (2,1) {$2$};
\node[above] at (3,1) {$3$};
\node[above] at (4,1) {$4$};
\foreach \x in {1,...,4}{
\fill (\x,1) circle (2pt);}

\draw[thick] (1,1) -- +(0,-1);
\draw[thick] (2,1) -- +(0,-1);
\draw[thick] (3,1) -- +(0,-1);
\draw[thick] (4,1) to[out=270,in=180] (4.5,0.5);
\end{tikzpicture}
\end{array}
\]
In the following we will often omit to indicate the vertex labeling on the top horizontal edge of the rectangle.
\end{ex}


\begin{lem}\label{lem:bijection_tableaux_cups}
There exists a bijection between the set of standard one-row bitableaux of shape $((k),(m-k))$ and the set of cup diagrams $\mathbb B^{((k),(m-k))}$. The map is given by sending a bitableau to the cup diagram whose left endpoints of cups or endpoints of half-cups are the vertices whose labels are contained in the second Young tableau.   
\end{lem}
\begin{proof}
Such a cup diagram exists and is unique (thus we have a well-defined map). The inverse is then the following: given a cup diagram, put the left endpoints of cups and half-cups in the right Young tableau of the bipartition $((k),(m-k))$ (in the unique order such that the bitableau will be standard).   
\end{proof}

\begin{ex}
We illustrate this bijection of cup diagrams with bitableaux for the bipartition $((2),(2))$. Explicitly the bijection is:
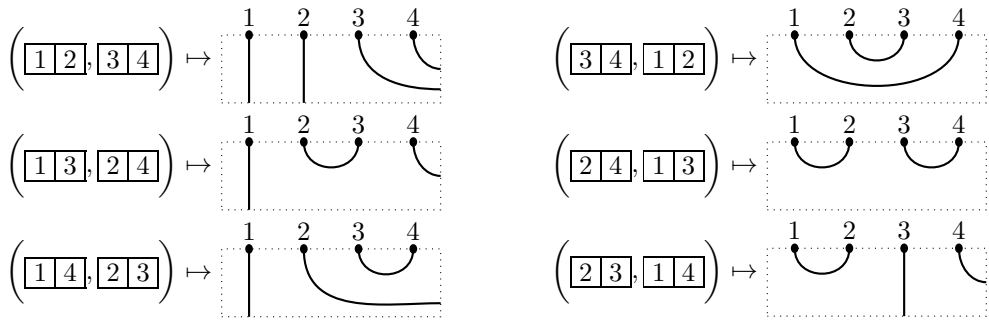
\begin{multicols}{2}
$\bigg(\young(12),\young(34)\bigg) \mapsto $
\begin{tikzpicture}[scale=0.9,xscale=0.8,yscale=1,baseline={(0,0.5)}]
\draw[dotted] (0.5,0) -- (0.5,1) -- (4.5,1) -- (4.5,0) -- cycle;
\node[above] at (1,1) {$1$};
\node[above] at (2,1) {$2$};
\node[above] at (3,1) {$3$};
\node[above] at (4,1) {$4$};

\foreach \x in {1,...,4}{
\fill (\x,1) circle (2pt);}

\draw[thick] (1,1) -- +(0,-1);
\draw[thick] (2,1) -- +(0,-1); 

\draw[thick] (3,1) to[out=270,in=180] (4.5,0.2);
\draw[thick] (4,1) to[out=270,in=180] (4.5,0.5);
\end{tikzpicture}

$\bigg(\young(13),\young(24)\bigg) \mapsto $
\begin{tikzpicture}[scale=0.9,xscale=0.8,yscale=1,baseline={(0,0.5)}]
\draw[dotted] (0.5,0) -- (0.5,1) -- (4.5,1) -- (4.5,0) -- cycle;
\node[above] at (1,1) {$1$};
\node[above] at (2,1) {$2$};
\node[above] at (3,1) {$3$};
\node[above] at (4,1) {$4$};
\foreach \x in {1,...,4}{
\fill (\x,1) circle (2pt);}

\draw[thick] (1,1) -- +(0,-1);
\draw[thick] (2,1) .. controls +(0,-.5) and +(0,-.5) .. (3,1);
\draw[thick] (4,1) to[out=270,in=180] (4.5,0.5);
\end{tikzpicture}

$\bigg(\young(14),\young(23)\bigg) \mapsto $
\begin{tikzpicture}[scale=0.9,xscale=0.8,yscale=1,baseline={(0,0.5)}]
\draw[dotted] (0.5,0) -- (0.5,1) -- (4.5,1) -- (4.5,0) -- cycle;
\node[above] at (1,1) {$1$};
\node[above] at (2,1) {$2$};
\node[above] at (3,1) {$3$};
\node[above] at (4,1) {$4$};

\foreach \x in {1,...,4}{
\fill (\x,1) circle (2pt);}
\draw[thick] (1,1) -- +(0,-1);
\draw[thick] (3,1) .. controls +(0,-.5) and +(0,-.5) .. (4,1);
\draw[thick] (2,1) to[out=270,in=180] (4.5,0.2);
\end{tikzpicture}

\columnbreak

$\bigg(\young(34),\young(12)\bigg) \mapsto $
\begin{tikzpicture}[scale=0.9,xscale=0.8,yscale=1,baseline={(0,0.5)}]
\draw[dotted] (0.5,0) -- (0.5,1) -- (4.5,1) -- (4.5,0) -- cycle;
\node[above] at (1,1) {$1$};
\node[above] at (2,1) {$2$};
\node[above] at (3,1) {$3$};
\node[above] at (4,1) {$4$};

\foreach \x in {1,...,4}{
\fill (\x,1) circle (2pt);}

\draw[thick] (1,1) .. controls +(0,-1) and +(0,-1) .. (4,1);
\draw[thick] (2,1) .. controls +(0,-0.5) and +(0,-0.5) .. (3,1);
\end{tikzpicture}

$\bigg(\young(24),\young(13)\bigg) \mapsto $
\begin{tikzpicture}[scale=0.9,xscale=0.8,yscale=1,baseline={(0,0.5)}]
\draw[dotted] (0.5,0) -- (0.5,1) -- (4.5,1) -- (4.5,0) -- cycle;
\node[above] at (1,1) {$1$};
\node[above] at (2,1) {$2$};
\node[above] at (3,1) {$3$};
\node[above] at (4,1) {$4$};

\foreach \x in {1,...,4}{
\fill (\x,1) circle (2pt);}

\draw[thick] (1,1) .. controls +(0,-0.5) and +(0,-0.5) .. (2,1);
\draw[thick] (3,1) .. controls +(0,-0.5) and +(0,-0.5) .. (4,1);
\end{tikzpicture}

$\bigg(\young(23),\young(14)\bigg) \mapsto $
\begin{tikzpicture}[scale=0.9,xscale=0.8,yscale=1,baseline={(0,0.5)}]
\draw[dotted] (0.5,0) -- (0.5,1) -- (4.5,1) -- (4.5,0) -- cycle;
\node[above] at (1,1) {$1$};
\node[above] at (2,1) {$2$};
\node[above] at (3,1) {$3$};
\node[above] at (4,1) {$4$};

\foreach \x in {1,...,4}{
\fill (\x,1) circle (2pt);}

\draw[thick] (3,1) -- +(0,-1);
\draw[thick] (1,1) .. controls +(0,-0.5) and +(0,-0.5) .. (2,1);
\draw[thick] (4,1) to[out=270,in=180] (4.5,0.5);

\end{tikzpicture}

\end{multicols}
\end{ex}

\begin{rem}\label{rem:dimension}
By Lemma~\ref{lem:bijection_tableaux_cups} and Proposition~\ref{prop:parametrization_of_components} the cup diagrams in $\mathbb B^{((k),(m-k))}$ parametrize the irreducible components of the exotic Springer fiber corresponding to the one-row bipartition $((k),(m-k))$. By Proposition~\ref{prop:parametrization_of_components} this exotic Springer fiber has dimension $m-k$ because $M\left((k)+(m-k)\right)$ is zero as the type A Springer fiber corresponding to the one-row partition $(k)+(m-k)=(m)$ is just a point. In particular, the dimension of this exotic Springer fiber can be read off the cup diagrams by counting the number of cups plus half-cups.    
\end{rem}

\begin{rem}\label{rem:diagrammatics_dictionary}
The diagrams in Definition~\ref{defi:one-boundary_cup_diagrams} naturally arise in the study of the one-boundary Temperley--Lieb algebra or blob algebra which is a quotient of a two-parameter Hecke algebra of type C,~\cite{MW03}. In order to explain this, let $q,z\in\mathbb C^*$ and let $[n]_q=\frac{q^n-q^{-n}}{q-q^{-1}}$ be the quantum integer. Following~\cite[\S2.1]{PRH14}, we define $b_m(q,z)$ as the unital associative $\mathbb C$-algebra generated by $U_1,\ldots,U_{m-1},e$ subject to the relations
\begin{alignat*}{3}
U_i^2&=-[2]_qU_i \hspace{1.8em} && \text{if }1\leq i\leq m-1 \\
U_iU_jU_i&=U_i && \text{if }|i-j|=1\\
U_iU_j&=U_jU_i && \text{if }|i-j|>1\\
U_{m-1}eU_{m-1}&=zU_{m-1} && \\
e^2&=e && \\
U_ie&=eU_i && \text{if }1\leq i\leq m-2. 
\end{alignat*}

The algebra $b_m(q,z)$ has a diagrammatic description,~\cite{NRG05}. The generators of $b_m(q,z)$ can be depicted as
\begin{equation}\label{eq:one-boundary_generators}
U_i=\begin{tikzpicture}[scale=.8,baseline={(0,-.55)}]
\draw[dotted] (-.3,-1.1) -- (4.3,-1.1) -- (4.3,0) -- (-.3,0) -- cycle;

\node at (1,-.55) {$\cdots$};
\node at (3.5,-.55) {$\cdots$};
\node[below] at (2,-1.05) {${}_i$};
\node[below] at (0,-1.05) {${}_1$};
\node[below] at (4,-1.05) {${}_m$};

\draw[thick] (0,0) -- (0,-1.1);
\draw[thick] (.5,0) -- (.5,-1.1);
\draw[thick] (1.5,0) -- (1.5,-1.1);

\draw[thick] (2,0) .. controls +(0,-.5) and +(0,-.5) .. +(.5,0);
\draw[thick] (2.5,-1.1) .. controls +(0,.5) and +(0,.5) .. +(-.5,0);
\draw[thick] (3,0) -- (3,-1.1);
\draw[thick] (4,0) -- (4,-1.1);
\end{tikzpicture}\hspace{1em} \left(1\leq i\leq m-1\right) \hspace{.6em},\hspace{1.2em}
e=\begin{tikzpicture}[scale=.8,baseline={(0,-.55)}]
\draw[dotted] (-.3,-1.1) -- (2.3,-1.1) -- (2.3,0) -- (-.3,0) -- cycle;

\node at (1,-.55) {$\cdots$};
\node[below] at (0,-1.05) {${}_1$};
\node[below] at (2,-1.05) {${}_m$};

\draw[thick] (0,0) -- (0,-1.1);
\draw[thick] (.5,0) -- (.5,-1.1);
\draw[thick] (1.5,0) -- (1.5,-1.1);
\draw[thick] (2,0) to[out=270,in=180] (2.3,-0.4);
\draw[thick] (2,-1.1) to[out=90,in=180] (2.3,-0.7);
\end{tikzpicture}\,\,,
\end{equation}
where the $U_i$ are just the usual pictures for the generators of the type A Temperley--Lieb algebra and the picture for $e$ consists of $m-1$ subsequent vertical rays followed by an opposing pair of half-cups. The multiplication is given by stacking the diagrams vertically. The second, third and sixth relations translate into isotopy relations and the first, fourth and fifth relation describe how to remove circles and half-circles. The picture for the generator $e$ explains the term {\it one-boundary Temperley--Lieb algebra} for $b_m(q,z)$. Alternatively, one can replace the picture for $e$ in~\eqref{eq:one-boundary_generators} by 
\[
e=\begin{tikzpicture}[scale=.8,baseline={(0,-.55)}]
\draw[dotted] (-.3,-1.1) -- (1.8,-1.1) -- (1.8,0) -- (-.3,0) -- cycle;

\fill (1.5,-0.55) circle (3.2pt);

\node at (0.5,-.55) {$\cdots$};

\draw[thick] (0,0) -- (0,-1.1);
\draw[thick] (1,0) -- (1,-1.1);
\draw[thick] (1.5,0) -- (1.5,-1.1);
\end{tikzpicture}\,\,,
\]
which consists of $m$ subsequent vertical strands with the rightmost strand being additionally decorated by a single blob. In this case, the algebra $b_m(q,z)$ is generally referred to as the {\it blob algebra},~\cite{MS94}. In terms of the blob diagrams, the fifth relation means that two blobs on a component can be merged into a single blob and the fourth relation explains how to remove circles with a blob.

By cutting a diagram obtained as the composition of some of the generators of $b_m(q,z)$ along its middle horizontal line, it can be separated uniquely into a top and bottom part (see, e.g.,~\cite[\S3]{NRG05} or~\cite[\S2.2]{MS94} for details). If we use the one-boundary diagrammatics, the top part consists of cups, half-cups and rays only. For the blob diagrammatics, the upper part consists of cups and rays only, some of which may be decorated with a single blob. For our purposes, we prefer to use the one-boundary cup diagrams since the decorated cup diagrams already appear in the description of the classical two-row Springer fibers of types C and D,~\cite{ES12},~\cite{Wil18}, and there appears to be a general overuse of dots or blobs in the literature on diagram algebras. 
\end{rem}

\section{Structure of the irreducible components}\label{sec:structure_irred_comp}

In this section, we provide an explicit description of the irreducible components of the exotic Springer fiber $\mathcal Fl^{((k),(m-k))}$ and deduce some important geometric properties (see Theorem~\ref{thm:algebraic_main_result}). 

\subsection{Embedding the exotic Springer fiber into a smooth variety} \label{sec:embedded_Springer_fiber}


Let $N>0$ be a large integer (see Remark~\ref{y_remark} for a more precise explanation of what is meant by ``large'') and let $z\colon\mathbb C^{2N}\to\mathbb C^{2N}$ be a nilpotent linear endomorphism with two equally sized Jordan blocks, i.e., there is a Jordan basis 
\begin{equation} \label{eq:Jordan_basis_of_z}
\begin{tikzpicture}[baseline={(0,0)}]
\node (e1) at (0,0) {$e_1$};
\node (e2) at (1,0) {$e_2$};
\node (space1) at (2,0) {$\ldots{}^{}$};
\node (eN) at (3.1,0) {$e_N$};
\node (f1) at (5,0) {$f_1$};
\node (f2) at (6,0) {$f_2$};
\node (space2) at (7,0) {$\ldots{}^{}$};
\node (fN) at (8.1,0) {$f_N$};
\path[->,font=\scriptsize,>=angle 90,bend right]
(e2) edge (e1)
(space1) edge (e2)
(eN) edge (space1)
(f2) edge (f1)
(space2) edge (f2)
(fN) edge (space2);
\end{tikzpicture}
\end{equation}
of $\mathbb C^{2N}$ on which $z$ acts as indicated (the vectors $e_1$ and $f_1$ are sent to zero). 



In the following, we will consider the smooth, projective variety 
\begin{equation} \label{eq:Y_i}
Y_m:=\big\{\left(F_1,\ldots,F_m\right) \mid F_i \subseteq \mathbb C^{2N},\, \dim F_i= i ,\, F_1\subseteq\ldots\subseteq F_m,\, zF_i\subseteq F_{i-1}\big\}
\end{equation}
defined in~\cite[\S2]{CK08}. 

\begin{rem} \label{y_remark}
Note that the conditions $zF_i\subseteq F_{i-1}$ yield
\begin{displaymath}
F_m \subseteq z^{-1}F_{m-1} \subseteq\ldots\subseteq z^{-m}(0) = \spn(e_1,\ldots,e_m,f_1,\ldots,f_m). 
\end{displaymath}
Thus, the variety $Y_m$ is independent of the choice of $N$ as long as $N \geq m$. In particular, we can always assume (by increasing $N$ if necessary) that all the subspaces of a flag in $Y_m$ are contained in the image of $z$. 
\end{rem}

Let $E_m\subseteq\mathbb C^{2N}$ be the subspace spanned by $e_1,\ldots,e_m,f_1,\ldots,f_m$. We equip $E_m$ with a symplectic form $\omega_m$ defined as follows: for all $j,j'\in\{1,\ldots,m\}$ we set 
\[
\omega_m(e_j,f_{j'})=\delta_{j+j',m+1}=-\omega_m(f_{j'},e_j).
\]
Otherwise, the pairing is defined to be zero. A direct computation shows that $\omega_m(z(v),w)=\omega_m(v,z(w))$ for all $v,w\in E_m$, i.e., the restriction $z_m$ of $z$ to the subspace $E_m\subseteq\mathbb C^{2N}$ is a nilpotent endomorphism which is self-adjoint with respect to $\omega_m$.

Thus, we can view the exotic Springer fiber $\mathcal Fl^{((k),(m-k))}$ as a subvariety of $Y_m$ via the following identification 
\begin{equation} \label{eq:type_D_embedding}
\mathcal Fl^{((k),(m-k))} \cong \left\{(F_1,\ldots,F_m) \in Y_m \left|\;\vcenter{\hbox{$e_k\in F_m\subseteq E_m$ and $F_m$ is}
																																					 \hbox{isotropic with respect to $\omega_m$}}
																														\right.\right\}.
\end{equation} 

The next proposition shows that the description of the embedding (\ref{eq:type_D_embedding}) can be simplified. More precisely, it turns out that we can completely ignore the symplectic structure for exotic Springer fibers corresponding to one-row bipartitions. 

\begin{prop}\label{prop:exotic_embedding}
The exotic Springer fiber $\mathcal Fl^{((k),(m-k))}$ can be viewed as a subvariety of $Y_m$ via the following identification
\begin{equation} \label{eq:exotic_embedding}
\mathcal Fl^{((k),(m-k))} \cong \left\{(F_1,\ldots,F_m) \in Y_m \mid e_k\in F_m\right\}.
\end{equation}
\end{prop}
\begin{proof}
Note that the inclusion $F_m\subseteq E_m$ is automatically satisfied by Remark~\ref{y_remark}. In order to reduce the description (\ref{eq:type_D_embedding}) to (\ref{eq:exotic_embedding}), it thus suffices to show that the isotropy condition on $F_m$ is automatically satisfied for all $(F_1,\ldots,F_m)\in Y_m$ and hence can be dropped, too.     

Let $(F_1,\ldots,F_m)\in Y_m$ be a flag. By induction we assume that we have already shown that the vector spaces $F_1,\ldots,F_i$ are isotropic with respect to $\omega_m$. Furthermore, we assume that there exist linearly independent vectors
\begin{eqnarray} \label{basis_induction}
\begin{array}{cc}
&e_1^{(i)},\hspace{.5em} e_2^{(i)}, \hspace{.5em}\ldots\hspace{.5em}, \hspace{.5em}e_{m-i-1}^{(i)}, \hspace{.5em}e_{m-i}^{(i)}, \\
&f_1^{(i)},\hspace{.5em} f_2^{(i)}, \hspace{.5em}\ldots\hspace{.5em}, \hspace{.5em}f_{m-i-1}^{(i)}, \hspace{.5em}f_{m-i}^{(i)},
\end{array}
\end{eqnarray}
in $\mathbb C^{2N}$ such that the following three properties hold:
\begin{enumerate}[{(\text{P}}1{)}]
\item \label{P1} The endomorphism $z$ maps each vector in (\ref{basis_induction}) to its left neighbor and the two leftmost vectors in each row to $F_i$.
\item \label{P2} We have a decomposition $F_i\oplus G_i=F_i^\perp$, where $G_i\subseteq\mathbb C^{2N}$ is the subspace spanned by the vectors in (\ref{basis_induction}) and the orthogonal complement is taken with respect to $\omega_m$. 
\item \label{P3} We have $\omega_m(e_j^{(i)},f_{j'}^{(i)})=\delta_{j+j',m-i+1}$.
\end{enumerate}

Note that the induction starts with the zero space $\{0\}$ and the basis $e_1,\ldots,e_m,f_1,\ldots,f_m$ of $E_m$ which satisfies~\ref{P1}-\ref{P3}.

Since $zF_{i+1}\subseteq F_i$, we can write
\[
F_{i+1}=F_i\oplus\spn(\alpha e_1^{(i)}+\beta f_1^{(i)})
\]
for some $\alpha,\beta \in \mathbb C$, $\alpha \neq 0$ or $\beta \neq 0$. Then properties~\ref{P2} and~\ref{P3} show that $F_{i+1}$ is indeed isotropic. 

In order to complete the inductive argument we need to show the existence of the vectors (\ref{basis_induction}) for the space $F_{i+1}$. We distinguish between two different cases depending on the parameters $\alpha$ and $\beta$.

{\it Case 1:} If $\alpha\neq 0$ and $\beta\neq 0$ we consider the linearly independent vectors
\begin{equation} \label{eq:Jordan_basis}
	\begin{split}
&\alpha e_1^{(i)}-\beta f_1^{(i)}\,\,\,,\,\,\, \alpha e_2^{(i)}-\beta f_2^{(i)}\,\,\,,\,\,\, \ldots \,\,\,,\,\,\, \alpha e_{m-i-1}^{(i)}-\beta f_{m-i-1}^{(i)}\\
&\alpha e_2^{(i)}+\beta f_2^{(i)}\,\,\,,\,\,\, \alpha e_3^{(i)}+\beta f_3^{(i)}\,\,\,,\,\,\, \ldots \,\,\,,\,\,\, \alpha e_{m-i}^{(i)}+\beta f_{m-i}^{(i)} 
	\end{split}
\end{equation}
and define 
\[
e_j^{(i+1)}:=\frac{1}{\sqrt{2\alpha\beta}}\left(\alpha e_j^{(i)}-\beta f_j^{(i)}\right) \hspace{1em} \text{and} \hspace{1em} f_j^{(i+1)}:=\frac{1}{\sqrt{2\alpha\beta}}\left(\alpha e_{j+1}^{(i)}+\beta f_{j+1}^{(i)}\right)
\]
for $j\in\{1,\ldots,m-i-1\}$. It is straightforward to check that this collection of vectors satisfies~\ref{P1}-\ref{P3} for $F_{i+1}$. 

{\it Case 2:} If $\beta=0$ we define
\[
e_j^{(i+1)}:=e_{j+1}^{(i)} \hspace{1em} \text{and} \hspace{1em} f_j^{(i+1)}:=f_j^{(i)}
\]
for $j\in\{1,\ldots,m-i-1\}$ and if $\alpha=0$ we set
\[
e_j^{(i+1)}:=e_j^{(i)} \hspace{1em} \text{and} \hspace{1em} f_j^{(i+1)}:=f_{j+1}^{(i)}
\]
for $j\in\{1,\ldots,m-i-1\}$. Again, one easily checks that this collection of vectors satisfies~\ref{P1}-\ref{P3} for $F_{i+1}$.
\end{proof}

\subsection{Explicit description of the irreducible components}\label{subsec:irred_components}
For the remainder of this article we will write $\mathcal Fl^{((k),(m-k))}$ to denote the embedded exotic Springer fiber via identification (\ref{eq:exotic_embedding}) from Proposition~\ref{prop:exotic_embedding}.

\begin{thm}\label{thm:algebraic_main_result}
The following statements hold:
\begin{enumerate}[(a)]
\item\label{second_part_thm} Given $\ba\in\mathbb B^{((k),(m-k))}$, let $K_\ba\subseteq Y_m$ be the subvariety consisting of all flags $(F_1,\ldots,F_m)\in Y_m$ which satisfy the following conditions imposed by the cup diagram $\ba$:
\begin{enumerate}[(i)]
\item If $i \CupConnect j$, then
\[
F_j=z^{-\frac{1}{2}(j-i+1)}F_{i-1}.
\]
\item If $\RayConnect$, then
\[
F_i=F_{i-1}+\spn\left(e_{\frac{1}{2}(i+\rho_\ba(i)))}\right),
\]
where $\rho_\ba(i)$ counts the number of rays to the left of vertex $i$ (including the vertex $i$) in $\ba$. 
\end{enumerate}
(There are no relations for vector spaces indexed by a vertex connected to a half-cup.) Then the variety $K_\ba\subseteq Y_m$ is an irreducible component of the exotic Springer fiber $\mathcal Fl^{((k),(m-k))}\subseteq Y_m$.
\item\label{third_part_thm} The irreducible component $K_\ba\subseteq\mathcal Fl^{((k),(m-k))}$ is an $(m-k)$-fold iterated fiber bundle over $\mathbb P^1$, i.e., there exist spaces $K_\ba=X_1,X_2,\ldots,X_{m-k},X_{m-k+1}=\mathrm{pt}$ together with maps $p_1,p_2,\ldots,p_{m-k}$ such that $p_j\colon X_j\to \mathbb P^1$ is a fiber bundle with typical fiber $X_{j+1}$. In particular, the irreducible component $K_\ba$ is smooth.
\item\label{first_part_thm} The map $\ba\mapsto K_\ba$ defines a bijection between the one-boundary cup diagrams in $\mathbb B^{((k),(m-k))}$ and the irreducible components of $\mathcal Fl^{((k),(m-k))}$.
\end{enumerate}
\end{thm}

Let $z_m$ be the restriction of $z$ to $E_m$. The following Lemma shows that we could replace $z$ by $z_m$ in relation (i) of Theorem~\ref{thm:algebraic_main_result}\ref{second_part_thm}. In particular, the description of the irreducible components of the exotic Springer fiber in Theorem~\ref{thm:algebraic_main_result}\ref{second_part_thm} also makes sense without the embedding into $Y_m$.

\begin{lem}
Given a flag $(F_1,\ldots,F_m)\in Y_m$, we have $z^{-\frac{1}{2}(j-i+1)}F_{i-1}=z_m^{-\frac{1}{2}(j-i+1)}F_{i-1}$ for all $i,j\in\{1,\ldots,m\}$, $i<j$. 
\end{lem}
\begin{proof}
It suffices to show that $z^{-\frac{1}{2}(j-i+1)}F_{i-1}$ is contained in $E_m$ for all $i,j\in\{1,\ldots,m\}$, $i<j$. This follows from the chain of inclusions
\begin{align*}
z^{-\frac{1}{2}(j-i+1)}F_{i-1}\subseteq z^{-(j-i+1)}F_{i-1}\subseteq z^{-(m-i+1)}F_{i-1}&\subseteq z^{-(m-i+1)}\left(z^{-(i-1)}(0)\right)\\ &=z^{-m}(0)=E_m.
\qedhere
\end{align*}
\end{proof}

%

For the proof of Theorem~\ref{thm:algebraic_main_result} we consider the subvariety $X^i_m \subseteq Y_m$, $i \in \{1,\ldots,m-1\}$, defined by 
\begin{equation}\label{eq:X_i}
X^i_m:=\{(F_1,\ldots,F_m) \in Y_m \mid F_{i+1}=z^{-1}F_{i-1}\},
\end{equation}
and the surjective morphism of algebraic varieties $q_m^i \colon X^i_m \twoheadrightarrow Y_{m-2}$ given by
\begin{equation}\label{eq:q_i_morphism}
(F_1,\ldots,F_m) \mapsto \left(F_1,\ldots,F_{i-1},zF_{i+2},\ldots,zF_m\right),
\end{equation}
see also~\cite[\S2]{CK08}.

\begin{lem}\label{lem:induction_alg_components}
Let $\ba\in\mathbb B^{((k),(m-k))}$ be a cup diagram with a cup connecting vertices $i$ and $i+1$ and let $\ba'\in\mathbb B^{((k-1),(m-k-1))}$ be the cup diagram obtained by deleting this cup. Then we have $K_\ba=(q_m^i)^{-1}(K_{\ba'})$. 
\end{lem}
\begin{proof}
We have to show that a flag $(F_1,\ldots,F_m)\in Y_m$ satisfies the conditions (i), (ii) from Theorem~\ref{thm:algebraic_main_result}\ref{second_part_thm} with respect to the cup diagram $\ba$ if and only if 
\[
q_m^i(F_1,\ldots,F_m)=(F_1,\ldots,F_{i-1},zF_{i+2},\ldots,zF_m)\in Y_{m-2}
\]
satisfies these conditions with respect to $\ba'$. 

In order to deal with condition (ii) suppose there is a ray connected to vertex $j$ in $\ba$. If $j\leq i-1$ there is nothing to show. If $j>i+1$ we need to show that
\[
F_j = F_{j-1}+\spn\left(e_{\frac{1}{2}(j-\alpha)}\right) \hspace{1.6em} \Leftrightarrow \hspace{1.6em} zF_j = zF_{j-1}+\spn\left(e_{\frac{1}{2}(j-2-\alpha)}\right).
\] 
But this follows by applying $z$ (resp.\ $z^{-1}$) to the left (resp.\ right) side of the equivalence.

For the proof involving condition (i) we refer to the proof of \cite[Lemma 3.2]{Weh09} for the precise argument. 
\end{proof}

\begin{prop}\label{prop:K_a_contained}
If $\ba\in\mathbb B^{((k),(m-k))}$, we have $K_\ba\subseteq\mathcal Fl^{((k),(m-k))}$.
\end{prop}

\begin{proof}
We proceed by induction on the number of cups in $\ba$. By Proposition~\ref{prop:exotic_embedding} it suffices to show that $e_k\in F_m$ for all $(F_1,\ldots,F_m)\in K_{\ba}$.

\medskip
{\it Induction Start.} Let $\ba\in\mathbb B^{((k),(m-k))}$ be the cup diagram without any cups contained in $\mathbb B^{((k),(m-k))}$, i.e., the one given by $k$ subsequent rays followed by $m-k$ half-cups:
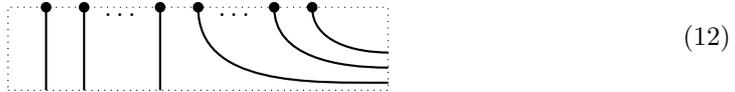
\begin{equation}\label{eq:induction_start_diagram}
\begin{tikzpicture}[baseline={(0,0.5)}, scale=1, xscale=1, yscale=1]
\draw[dotted] (0.5,-0.1) -- (0.5,1) -- (5.5,1) -- (5.5,-0.1) -- cycle;

\draw[thick] (1,1) -- +(0,-1.1);
\draw[thick] (1.5,1) -- +(0,-1.1);
\draw[thick] (2.5,1) -- +(0,-1.1);

\draw[thick] (3,1) to [out=270,in=180](5.5,0);
\draw[thick] (4,1) to [out=270,in=180](5.5,0.2);
\draw[thick] (4.5,1) to [out=270,in=180](5.5,0.4);

\foreach \x in {1,1.5,2.5,3,4,4.5}{
\fill (\x,1) circle (2pt);}

\node at (2,0.9) {\ldots};
\node at (3.5,0.9) {\ldots};
\end{tikzpicture}
\end{equation}

Given a flag $(F_1,\ldots,F_m)\in K_\ba$ it follows from the relations in Theorem~\ref{thm:algebraic_main_result}\ref{second_part_thm} that $F_k=\spn(e_1,\ldots,e_k)$. In particular, we have $e_k\in F_k\subseteq F_m$ which implies $(F_1,\ldots,F_m)\in\mathcal Fl^{((k),(m-k))}$. 

\medskip
{\it Inductive Step.} Assume there exists a cup in $\ba\in\mathbb B^{((k),(m-k))}$ and fix a cup connecting neighboring vertices $i$ and $i+1$. Let $\ba'\in\mathbb B^{((k-1),(m-k-1))}$ denote the cup diagram obtained by removing this cup. 

Given an arbitrary flag $(F_1,\ldots,F_m)\in K_\ba$, we have 
\[
q_m^i(F_1,\ldots,F_m)=\left(F_1,\ldots,F_{i-1},zF_{i+2},\ldots,zF_m\right)\in K_{\ba'}
\]
by Lemma~\ref{lem:induction_alg_components}. By induction we assume that $K_{\ba'}\subseteq\mathcal Fl^{((k-1),(m-k-1))}$. In particular, we have $e_{k-1}\in zF_m$, i.e., $e_{k-1}=z(v)$ for some $v\in F_m$. Any such $v\in F_m$ must be of the form $v=e_k+\alpha e_1+\beta f_1$ for $\alpha,\beta\in\mathbb C$. Since $\spn(e_1,f_1)=\ker(z)=z^{-1}(0)\subseteq z^{-1}F_{i-1}=F_{i+1}\subseteq F_m$ we see that $v\in F_m$ implies $e_k\in F_m$. But $e_k\in F_m$ means $(F_1,\ldots,F_m)\in\mathcal Fl^{((k),(m-k))}$. 
\end{proof}

\begin{proof}[Proof (Theorem~\ref{thm:algebraic_main_result}).]
We first show that $K_\ba$ is an $(m-k)$-fold iterated $\mathbb P^1$-bundle. This follows as in \cite[Proposition 5.1]{Fun03} (see also \cite[Section 8]{Sch12}). We briefly recall the argument in our setup. Let $i_1<i_2<\ldots<i_{m-k}$ denote the vertices connected to either a left endpoint of a cup or a half-cup in $\ba$. Note that the space $F_{i_1-1}$ is the same for every flag $(F_1,\ldots,F_m)\in K_\ba$ because each vertex to the left of $i$ is connected to a ray and successively applying relation (ii) in Theorem~\ref{thm:algebraic_main_result}~\ref{second_part_thm} determines this space uniquely. Hence, we can define the fiber bundle
\[
p_1\colon K_\ba\to \mathbb P(z^{-1}F_{i_1-1}/F_{i_1-1})\cong\mathbb P^1\,,\,\,(F_1,\ldots,F_m)\mapsto F_{i_1}/F_{i_1-1}
\] 
which has typical fiber $X_2$ consisting of all flags in $(F_1,\ldots,F_m)\in K_\ba$ with $F_{i_1}$ (and $F_j$, if $(i_1+1)$ and $j$ are connected by a cup) fixed. Now we repeat the above construction replacing $X_1$ by $X_2$ and the vertex $i_1$ by $i_2$ and continue until we have exhausted all of the vertices $i_1<\ldots<i_{m-k}$.

In order to prove parts~\ref{second_part_thm} and~\ref{third_part_thm} of Theorem~\ref{thm:algebraic_main_result} we first note that $K_\ba$ is smooth and connected since it is an $(m-k)$-fold iterated fiber bundle over $\mathbb P^1$. 
In particular, $K_\ba$ is irreducible (see, e.g.,~\cite[p.\ 21]{GH94}). By Proposition~\ref{prop:K_a_contained} the variety $K_\ba$ is contained in $\mathcal Fl^{((k),(m-k))}$ and its dimension equals $m-k$ which is the dimension of $\mathcal Fl^{((k),(m-k))}$ by~\cite[Theorem 2.12]{NRS16}. Hence, $K_\ba$ is an irreducible component of the (embedded) exotic Springer fiber $\mathcal Fl^{((k),(m-k))}$. This proves parts~\ref{second_part_thm} and~\ref{third_part_thm} of Theorem~\ref{thm:algebraic_main_result}. 

By Proposition~\ref{prop:parametrization_of_components} and Lemma~\ref{lem:bijection_tableaux_cups} we know that the cup diagrams in $\mathbb B^{((k),(m-k))}$ are in bijective correspondence with the irreducible components of $\mathcal Fl^{((k),(m-k))}$. Since the varieties $K_\ba$ are different for different $\ba\in\mathbb B^{((k),(m-k))}$, we see that the map $\ba\mapsto K_\ba$ explicitly realizes this bijection which proves part~\ref{first_part_thm} of the theorem.
\end{proof}

\begin{ex}
For the cup diagrams from Example~\ref{ex:cup_diagrams}, Theorem~\ref{thm:algebraic_main_result}\ref{second_part_thm} yields the following sets of flags
\begin{itemize}
\item $K_\ba=\big\{F_1 \subseteq\spn(e_1,f_1)\subseteq\spn(e_1,e_2,f_1)\subseteq\spn(e_1,e_2,e_3,f_1)\big\}$,
\item $K_\bb=\big\{\spn(e_1)\subseteq F_2\subseteq\spn(e_1,e_2,f_1)\subseteq\spn(e_1,e_2,e_3,f_1)\big\}$,
\item $K_{\bf c}=\big\{\spn(e_1)\subseteq\spn(e_1,e_2)\subseteq F_3\subseteq\spn(e_1,e_2,e_3,f_1)\big\}$,
\item $K_{\bf d}=\big\{\spn(e_1)\subseteq\spn(e_1,e_2)\subseteq\spn(e_1,e_2,e_3)\subseteq F_4\big\}$,
\end{itemize}
each of which is an irreducible component of $\mathcal Fl^{((3),(1))}$ isomorphic to $\mathbb P^1$.
\end{ex}

\section{Topology of exotic Springer fibers}\label{sec:Topological_Springer_fiber}

The purpose of this section is to define a simple topological model for the exotic Springer fiber $\mathcal Fl^{((k),(m-k))}$, i.e., we construct a topological space $\mathcal S^{((k),(m-k))}$ homeomorphic to $\mathcal Fl^{((k),(m-k))}$ which can be described explicitly using the (one-boundary) cup diagrams (see Theorem~\ref{thm:main_result_1}). As an application, we obtain a diagrammatic description of the pairwise intersections of the irreducible components of $\mathcal Fl^{((k),(m-k))}$ (see Theorem~\ref{thm:intersection_of_components}).  

\subsection{A topological model}

Let $\mathbb S^2=\{(a,b,c)\in\mathbb R^3\mid a^2+b^2+c^2=1\}\subseteq \mathbb R^3$ be the two-dimensional standard unit sphere with north pole $p=(0,0,1)$. Given a cup diagram $\ba\in\mathbb B^{((k),(m-k))}$, define 
\[
S_\ba=\big\{(x_1,\ldots,x_m)\in\left(\mathbb S^2\right)^m\mid x_j=-x_i\text{ if }i\CupConnect j,\text{ and }x_i=p\text{ if }\RayConnect\big\}.
\]
Note that there are no relations involving coordinates indexed by a vertex connected to a half-cup. Note that $S_\ba$ is homeomorphic to a product of two-spheres in which each cup and half-cup in $\ba\in\mathbb B^{((k),(m-k))}$ contributes exactly one sphere. The $((k),(m-k))$ {\it topological exotic Springer fiber} $\mathcal S^{((k),(m-k))}$ is defined as the union 
\[
\mathcal S^{((k),(m-k))} := \bigcup_{\ba \in \mathbb B^{((k),(m-k))}}S_\ba \subseteq\left(\mathbb S^2\right)^m.
\]

In order to facilitate the comparison of $\mathcal S^{((k),(m-k))}$ and $\mathcal Fl^{((k),(m-k))}$, we first rewrite the definition of $S_\ba$ as a submanifold of a product of projective spaces instead of two-spheres (see Lemma~\ref{lem:sphere_vs_projective_space} below). Let $e,f$ be the standard basis of $\mathbb C^2$. Consider the stereographic projection 
\[
\sigma\colon\mathbb S^2\backslash\{p\} \to \mathbb C; \,\hspace{.1cm}(a,b,c) \mapsto \frac{a}{1-c}+\frac{b}{1-c}\sqrt{-1},
\]
and its analog for projective space
\[
\theta\colon\mathbb P^1\backslash\{\spn(e)\} \to \mathbb C; \,\hspace{.1cm}\spn(\alpha e+\beta f) \mapsto \frac{\alpha}{\beta}. 
\]
We can use $\sigma$ and $\theta$ to define a diffeomorphism $\gamma\colon\mathbb P^1 \to \mathbb S^2$ by
\[
\spn(\alpha e+\beta f) \mapsto \begin{cases}
  \sigma^{-1}\left(\theta\left(\spn(\alpha e+\beta f)\right)\right) & \text{if }\spn(\alpha e+\beta f)\neq \spn(e),\\
  p=(0,0,1) & \text{if }\spn(\alpha e+\beta f)=\spn(e).
\end{cases}
\]
This induces a diffeomorphism $\gamma_m \colon\left(\mathbb P^1\right)^m \to \left(\mathbb S^2\right)^m$ on the $m$-fold products by setting $\gamma_m(l_1,\ldots,l_m):=\left(\gamma(l_1),\ldots,\gamma(l_m)\right)$.
 
\begin{lem} \label{lem:sphere_vs_projective_space}
There is an equality of sets $\gamma_m(T_\ba)=S_\ba$ for every $\ba\in\mathbb B^{((k),(m-k))}$, where $T_\ba\subseteq\left(\mathbb P^1\right)^m$ is defined by
\[
T_\ba:=\big\{(l_1,\ldots,l_m)\in\left(\mathbb P^1\right)^m \mid l^\perp_i=l_j \text{ if }i\CupConnect j\text{ and }l_i=\spn(e) \text{ if }\text{ if }\RayConnect\big\},
\]
where the orthogonal complement is taken with respect to the standard Hermitian form on $\mathbb C^2$ defined by declaring the standard basis $e,f$ to be an orthonormal basis. 
\end{lem}
\begin{proof} This follows from the definition of $\gamma_m$. We omit the details.
\end{proof}

\subsection{An explicit homeomorphism}

Recall the definition of the variety $Y_m$ in (\ref{eq:Y_i}) from the beginning of Section~\ref{sec:structure_irred_comp}. In order to state the next proposition we equip $\mathbb C^{2N}$ with a Hermitian form by declaring the Jordan basis $e_1,\ldots,e_N,f_1,\ldots,f_N$ in (\ref{eq:Jordan_basis_of_z}) to be an orthonormal basis. Moreover, let $C \colon \mathbb C^{2N} \to \mathbb C^2$ be the linear map defined by $C(e_i)=e$ and $C(f_i)=f$, $i \in \{1,\ldots,N\}$. The following proposition allows us to compare the submanifold $T_\ba$ from Lemma~\ref{lem:sphere_vs_projective_space} with the irreducible component $K_\ba$ of $\mathcal Fl^{((k),(m-k))}$ from Theorem~\ref{thm:algebraic_main_result}.

\begin{prop}[{\cite[Theorem 2.1]{CK08}}] \label{homeo}
The map $\phi_m: Y_m \rightarrow (\mathbb P^1)^m$ given by the assignment
\[
(F_1,\ldots,F_m) \mapsto \left(C(F_1),C(F_2 \cap F_1^\perp),\ldots,C(F_m \cap F_{m-1}^\perp)\right)
\]
is a diffeomorphism, where the orthogonal complement is taken with respect to the Hermitian form of $\mathbb C^{2N}$. 
\end{prop}

Using the map $\phi_m$ we can now state the following key proposition.

\begin{prop} \label{prop:preimage_contained}
We have $K_\ba=\phi_m^{-1}(T_\ba)$.
\end{prop}

In order to prove Proposition~\ref{prop:preimage_contained}, we first note the following lemma. (Recall the definitions of the variety $X_i^m\subseteq Y_m$ from (\ref{eq:X_i}) and the morphism of algebraic varieties $q_m^i$ from~(\ref{eq:q_i_morphism})). 

\begin{lem}
The diffeomorphism $\phi_m$ maps $X^i_m$ bijectively to the set
\begin{equation} \label{eq:defi_A}
A_m^i:=\{(l_1,\ldots,l_m)\in(\mathbb P^1)^m\mid l_{i+1}=l_i^\perp\}.
\end{equation}
Moreover, we have a commutative diagram
\begin{equation} \label{eq:comm_diag}
\begin{tikzcd} 
X^i_m \arrow[two heads]{rr}{q_m^i} \arrow{d}{\cong}[left]{\phi_{m}\vert_{X_m^i}} &&Y_{m-2} \arrow{d}{\phi_{m-2}}[left]{\cong}\\
A_m^i \arrow[two heads]{rr}{f^{i}_{m}\vert_{A_m^i}} &&\left(\mathbb P^1\right)^{m-2} 
\end{tikzcd}
\end{equation}
where $f^i_m\colon\left(\mathbb P^1\right)^m \twoheadrightarrow\left(\mathbb P^1\right)^{m-2}$ is the map which forgets coordinates $i$ and $i+1$. The orthogonal complement in (\ref{eq:defi_A}) is taken with respect to the Hermitian structure of $\mathbb C^2$.

Furthermore, we have a commutative diagram
\begin{equation} \label{eq:commutative_diag}
\begin{tikzcd} 
Y_m \arrow[two heads]{rr}{\psi_m} \arrow{d}{\cong}[left]{\phi_m} &&Y_{m-1} \arrow{d}{\phi_{m-1}}[left]{\cong}\\
\left(\mathbb P^1\right)^m \arrow[two heads]{rr}{g_m} &&\left(\mathbb P^1\right)^{m-1} 
\end{tikzcd}
\end{equation}
where $g_m$ (resp.\ $\psi_m$) is the map which forgets the last coordinate (resp.\ vector space of the flag).
\end{lem}
\begin{proof}
For the proof of the commutativity of the first diagram (\ref{eq:comm_diag}) we refer to \cite[Theorem 2.1]{CK08} or \cite[Lemma 2.4]{Weh09}. The commutativity of the second diagram (\ref{eq:commutative_diag}) follows directly from the definitions of the respective maps.
\end{proof}

\begin{proof}[Proof (Proposition~\ref{prop:preimage_contained}).]
We prove the proposition by induction on the number of cups in $\ba$.

\smallskip
{\it Base of Induction.} We first prove Proposition~\ref{prop:preimage_contained} for all cup diagrams $\ba\in\mathbb B^{((k),(m-k))}$ as in (\ref{eq:induction_start_diagram}) consisting of rays and half-cups only. 

We first consider the cup diagram $\ba\in\mathbb B^{((m),(0))}$ consisting of $m$ successive rays. In this case $T_\ba$ consists of a single element $(\spn(e),\ldots,\spn(e))\in\left(\mathbb P^1\right)^m$ whose preimage $(F_1,\ldots,F_m)$ under the diffeomorphism $\phi_m$ is given by the flag 
\[ 
(\spn(e_1),\spn(e_1,e_2),\ldots,\spn(e_1,\ldots,e_m))
\] 
since 
\begin{align*}
\spn(e)&=C(\spn(e_i))=C\left(\spn(e_1,\ldots,e_i)\cap\spn(e_i,\ldots,e_m,f_1,\ldots,f_m)\right)\\
			 &=C(F_i\cap F_{i-1}^\perp).
\end{align*}
The relations in part (ii) of Theorem~\ref{thm:algebraic_main_result} show that this flag is indeed the unique element in $K_\ba$. 

Assume that there exists a half-cup in $\ba\in\mathbb B^{((k),(m-k))}$ (otherwise we are done by the above). Let $\ba'\in\mathbb B^{((k),(m-k-1))}$ be the cup diagram obtained by deleting the half-cup connected to the rightmost vertex $m$. Thus, we have
\begin{equation} \label{eq:preimage_under_q}
K_\ba=\psi_m^{-1}\left(K_{\ba'}\right)=\psi_m^{-1}\left(\phi_{m-2}^{-1}(T_{\ba'})\right)=\phi_m^{-1}\left(g_m^{-1}\left(T_{\ba'}\right)\right)=\phi_m^{-1}(T_\ba),
\end{equation} 
where the first equality is evident, the second one follows from the inductive hypothesis $\phi_{m-2}^{-1}(T_{\ba'})=K_{\ba'}$, the third one follows from the commutativity of the diagram (\ref{eq:commutative_diag}), and the last one is the obvious fact that $g_m^{-1}\left(T_{\ba'}\right)=T_\ba$.

\smallskip
{\it Inductive Step.} Now we assume that there exists at least one cup in $\ba$. Thus, there exists a cup connecting neighboring vertices, say $i$ and $i+1$ (because the diagram is crossingless). Let ${\ba{''}}\in\mathbb B^{((k-1),(m-k-1))}$ be the cup diagram obtained by removing this cup.

By an analogous argument as in the base case of the induction we obtain 
\begin{equation} 
K_\ba=\left(q_m^i\right)^{-1}\left(K_{\ba{''}}\right)=\left(q_m^i\right)^{-1}\left(\phi_{m-2}^{-1}(T_{\ba{''}})\right)=\phi_m^{-1}\left(\left(f_m^i\right)^{-1}\left(T_{\ba{''}}\right)\right)=\phi_m^{-1}(T_\ba),
\end{equation} 
where the first equality follows from Lemma~\ref{lem:induction_alg_components}, the second one from the inductive hypothesis, the third one from the commutativity of the diagram (\ref{eq:comm_diag}), and the last one follows from $\left(f_m^i\right)^{-1}\left(T_{\ba{''}}\right)=T_\ba$ which is evident.
\end{proof}

\begin{thm}\label{thm:main_result_1}
The diffeomorphism $\left(\mathbb S^2\right)^m\xrightarrow{\gamma_m^{-1}}\left(\mathbb P^1\right)^m\xrightarrow{\phi_m^{-1}} Y_m$ restricts to a homeomorphism
\[
\mathcal S^{((k),(m-k))}\xrightarrow\cong\mathcal Fl^{((k),(m-k))}
\]
such that the images of the $S_\ba$ under this homeomorphism are precisely the irreducible components $K_\ba$ of $\mathcal Fl^{((k),(m-k))}$ for all $\ba\in\mathbb B^{((k),(m-k))}$. 
\end{thm}
\begin{proof}
The image of $\mathcal S^{((k),(m-k))}\subseteq\left(\mathbb S^2\right)^m$ under the diffeomorphism $\phi_m^{-1}\circ\gamma_m^{-1}$ is given by 
\begin{eqnarray*}
\phi_m^{-1}\left(\gamma_m^{-1}\left(\mathcal S^{((k),(m-k))}\right)\right) 
&=&\bigcup_{\ba \in\mathbb B^{((k),(m-k))}}\phi_m^{-1}\left(\gamma_m^{-1}\left(S_\ba\right) \right)\\ &=&\bigcup_{\ba \in\mathbb B^{((k),(m-k))}}\phi_m^{-1}\left(T_\ba\right) \qquad \text{(by Lemma \ref{lem:sphere_vs_projective_space})} \\
&=&\bigcup_{\ba \in\mathbb B^{((k),(m-k))}}K_\ba.  \qquad \text{(by Proposition~\ref{prop:preimage_contained})}
\end{eqnarray*}
The last line of the above chain of equalities equals $\mathcal Fl^{((k),(m-k))}$ by Theorem~\ref{thm:algebraic_main_result}.

\end{proof}

\begin{ex}\label{ex:topological_Springer_fiber}
The submanifolds of $\left(\mathbb S^2\right)^4$ associated with the cup diagrams in $\mathbb B^{{((3),(1))}}$ (see Example~\ref{ex:cup_diagrams}) are the following:

\medskip

\begin{minipage}{6cm}
\begin{itemize}
\item $S_\ba=\{(x,-x,p,p)\mid x\in\mathbb S^2\}$,
\item $S_\bb=\{(p,x,-x,p)\mid x\in\mathbb S^2\}$,
\end{itemize}
\end{minipage}
\begin{minipage}{6cm}
\begin{itemize}
\item $S_{\bf c}=\{(p,p,x,-x)\mid x\in\mathbb S^2\}$,
\item $S_{\bf d}=\{(p,p,p,x)\mid x\in\mathbb S^2\}$.
\end{itemize}
\end{minipage}

\medskip

\noindent Each of these manifolds is homeomorphic to a two-sphere. Their pairwise intersection is either a point or empty, e.g.\ we have $S_\ba\cap S_\bb=\{(p,-p,p,p)\}$ and $S_\ba\cap S_{\bf c}=\emptyset$ (see also Theorem~\ref{thm:intersection_of_components} below for a more general statement). The ${((3),(1))}$ topological exotic Springer fiber is a chain of four two-spheres:
\[
\mathcal S^{((3),(1))}\hspace{.4em}\cong\hspace{.4em}
\begin{tikzpicture}[baseline={(0,-.1)}, scale=.6]
\draw[thick] (0,0) circle (1.1);
\draw[thick] (2.2,0) circle (1.1);
\draw[thick] (4.4,0) circle (1.1);
\draw[thick] (6.6,0) circle (1.1);
\draw[densely dotted,thick] (0,0) ellipse (1.1 and 0.25);
\draw[densely dotted,thick] (2.2,0) ellipse (1.1 and 0.25);
\draw[densely dotted,thick] (4.4,0) ellipse (1.1 and 0.25);
\draw[densely dotted,thick] (6.6,0) ellipse (1.1 and 0.25);
\node at (2.2,.6) {$S_{\bf b}$};
\node at (0,.6) {$S_{\bf a}$};
\node at (4.4,.6) {$S_{\bf c}$};
\node at (6.6,.6) {$S_{\bf d}$};
\end{tikzpicture}
\]
More generally, it is straightforward to check that the ${((m-1),(1))}$ topological exotic Springer fiber is chain of $m$ two-spheres.
\end{ex}

\subsection{Intersections of irreducible components}

Next, we give a diagrammatic description of the pairwise intersections of the submanifolds $S_\ba\subseteq\left(\mathbb S^2\right)^m$, $\ba\in \mathbb B^{((k),(m-k))}$. Our result (see Theorem~\ref{thm:intersection_of_components}) is formulated using so-called circle diagrams.  

\begin{defi}
Let $\ba,\bb\in\mathbb B^{((k),(m-k))}$. The {\it circle diagram} $\overline{\ba}\bb$ is defined as the diagram obtained by reflecting the diagram $\ba$ in its horizontal middle line and then sticking the resulting diagram, denoted by $\overline{\ba}$, on top of the cup diagram $\bb$, i.e., we glue the two diagrams along the horizontal line containing the vertices (thereby identifying the vertices of $\overline{\ba}$ and $\bb$ pairwise from left to right).
\end{defi}

\begin{rem}
In general, the diagram $\overline{\ba}\bb$ consists of several connected components each of which is either a circle (i.e., a closed connected component) or a line segment. Note that there are three types of line segments: lines containing two half-cups, lines containing two rays, and lines containing one ray and a half-cup. A line which contains a half-cup or ray of $\ba$ and a half-cup or ray of $\bb$ is called a {\it propagating line}, i.e., the endpoints of the line are on opposite sides with respect to the horizontal middle line of $\overline{\ba}\bb$.  
\end{rem}

\begin{ex}
Here is an example illustrating the gluing of two cup diagrams in order to obtain a circle diagram:
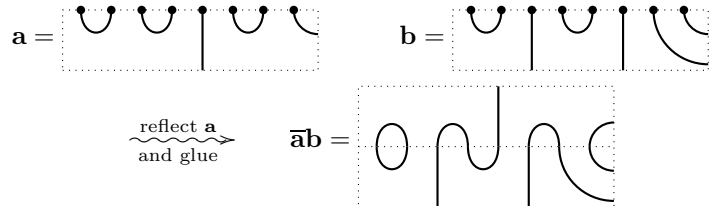
\[
\ba=
\begin{tikzpicture}[scale=.8,baseline={(0,-.45)}]
\draw[thick,rotate around={180:(2,0)}] (0.5,0) arc (0:90:4mm); 

\draw[dotted] (-0.3,-1) -- (3.9,-1) -- (3.9,0) -- (-0.3,0) -- cycle;

\draw[thick] (0,0) .. controls +(0,-0.5) and +(0,-0.5) .. +(0.5,0);
\draw[thick] (1,0) .. controls +(0,-0.5) and +(0,-0.5) .. +(0.5,0);
\draw[thick] (2.5,0) .. controls +(0,-0.5) and +(0,-0.5) .. +(0.5,0);
\draw[thick] (2,0) -- +(0,-1);

\foreach \x in {0,0.5,1,...,3.5}{
\fill (\x,0) circle (2pt);}

\end{tikzpicture}
\hspace{3em}
\bb=
\begin{tikzpicture}[scale=.8,baseline={(0,-0.45)}]
\draw[dotted] (-0.3,-1) -- (3.9,-1) -- (3.9,0) -- (-0.3,0) -- cycle;
\draw[thick,rotate around={180:(2,0)}] (0.5,0) arc (0:90:4mm); 
\draw[thick,rotate around={180:(2,0)}] (1,0) arc (0:90:9mm); 

\draw[thick] (0,0) .. controls +(0,-0.5) and +(0,-0.5) .. +(0.5,0);
\draw[thick] (1.5,0) .. controls +(0,-0.5) and +(0,-0.5) .. +(0.5,0);
\draw[thick] (2.5,0) -- +(0,-1);
\draw[thick] (1,0) -- +(0,-1);
\foreach \x in {0,0.5,1,...,3.5}{
\fill (\x,0) circle (2pt);}
\end{tikzpicture}
\]
\[
\begin{xy}
	\xymatrix@=1em{
\ar@{~>}[rrr]^{\textrm{reflect }\ba}_{\textrm{and glue}}  &&&
}
\end{xy}
\hspace{1.5em}
\overline{\ba}\bb=
\begin{tikzpicture}[scale=.8,baseline={(0,0)}]
\draw[dotted] (-.3,-1) -- (3.9,-1) -- (3.9,1) -- (-.3,1) -- cycle;
\draw[dotted] (-.3,0) -- (3.9,0);

\draw[thick] (0,0) .. controls +(0,.5) and +(0,.5) .. +(.5,0);
\draw[thick] (1,0) .. controls +(0,.5) and +(0,.5) .. +(.5,0);
\draw[thick] (2.5,0) .. controls +(0,.5) and +(0,.5) .. +(.5,0);
\draw[thick] (2,0) -- +(0,1);

\draw[thick,rotate around={180:(2,0)}] (.5,0) arc (0:-90:4mm);
\draw[thick,rotate around={180:(2,0)}] (.5,0) arc (0:90:4mm); 
\draw[thick,rotate around={180:(2,0)}] (1,0) arc (0:90:9mm); 

\draw[thick] (0,0) .. controls +(0,-.5) and +(0,-.5) .. +(.5,0);
\draw[thick] (1.5,0) .. controls +(0,-.5) and +(0,-.5) .. +(.5,0);
\draw[thick] (2.5,0) -- +(0,-1);
\draw[thick] (1,0) -- +(0,-1);
\end{tikzpicture}
\]
Hence, $\overline{\ba}\bb$ consists of one circle and three line segments. Note that the line containing to rays and the line containing two half-cups are both propagating. The line which contains a ray and a half-cup is an example of a non-propagating line.
\end{ex}

\begin{thm}\label{thm:intersection_of_components}
Let $\ba,\bb\in \mathbb B^{((k),(m-k))}$. We have $S_{\ba}\cap S_{\bb}\neq\emptyset$ if and only if all lines containing two rays in $\overline{\ba}\bb$ are propagating. Moreover, if $S_{\ba}\cap S_{\bb}\neq\emptyset$, we have a homeomorphism $S_{\ba}\cap S_{\bb}\cong\left(\mathbb S^2\right)^k$, where $k$ counts the number of circles plus the number of line segments in $\overline{\ba}\bb$ containing two half-cups.
\end{thm}

\begin{proof}
Assume that all lines in $\overline{\ba}\bb$ containing two rays are propagating. Given a circle or a line containing two half-cups in $\overline{\ba}\bb$, let $i$ denote the leftmost vertex (i.e., $i$ is the smallest element) on this component. Choose an element $x_i\in\left(\mathbb S^2\right)^m$. This determines all other $x_j$ for $j$ in the component by applying the relation $x_k=-x_l$ for the cups on the circle. Note that a circle always consists of an even number of cups. Hence, going around the circle once therefore yields an even number of sign changes and the thus the relations are consistent. For line segments containing two half-cups there is nothing to show. Given a line, we pick a vertex $i$ connected to a ray and set $x_i=p$. The remaining coordinates in the component are then determined. Note that for a given line the coordinate relations are consistent because the line is propagating and thus there is an even number of cups between the two rays. For a line segment containing a half-cup the relations are always consistent. Projecting onto the leftmost endpoints of circles and line segments containing two half-cups yields the homeomorphism $S_{\ba}\cap S_{\bb}\cong\left(\mathbb S^2\right)^k$.

On the other hand, if there is a non-propagating line containing two rays in $\overline{\ba}\bb$, there is an odd number of cups connecting the two rays of the line. Hence the relations coming from the cups are not consistent with the ones coming from the rays. Hence, the intersection $S_{\ba}\cap S_{\bb}$ must be empty.       
\end{proof}

\begin{rem}\label{rem:generalization}
We would like to make two remarks.
\begin{enumerate}
\item By using Theorem~\ref{thm:main_result_1}, we see that Theorem~\ref{thm:intersection_of_components} provides a description of the topology of the pairwise intersections of the irreducible components $K_\ba$ of the exotic Springer fiber $\mathcal Fl^{((k),(m-k))}$.
\item The argument in the proof of Theorem~\ref{thm:intersection_of_components} can be applied to $S_{\ba}$ and $S_{\bb}$ even if $\ba\in \mathbb B^{((k),(m-k))}$ and $\bb\in\mathbb B^{((l),(m-l))}$ for $k\neq l$. That is, it can be used to study the intersections of two irreducible components contained in two different exotic Springer fibers. This will be important in Section~\ref{sec:torus_fixed_points} (in particular, the proof of Theorem~\ref{thm:intersection_of_attracting_cells}). 
\end{enumerate}
\end{rem}

\section{$\mathbb C^*$-action and Bia\l{}ynicki-Birula paving}\label{sec:torus_fixed_points}

There is a $\mathbb C^*$-action on $\mathbb C^{2N}$ given by $t.e_i=te_i$ and $t.f_i=t^{-1}f_i$ which yields a $\mathbb C^*$-action on the variety $Y_m$ defined in~(\ref{eq:Y_i}). This action restricts to the embedded exotic Springer fiber $\mathcal Fl^{((k),(m-k))}\subseteq Y_m$ (recall Proposition~\ref{prop:exotic_embedding} for this embedding). In this section, we study the isolated fixed points under this $\mathbb C^*$-action and explicitly describe the corresponding attracting cells and the resulting Bia\l{}ynicki-Birula decomposition of the exotic Springer fiber $\mathcal Fl^{((k),(m-k))}$. 

%

\subsection{$\mathbb C^*$-fixed points} In order to describe the isolated fixed points of the $\mathbb C^*$-action on the exotic Springer fiber we introduce the notion of a combinatorial weight.

\begin{defi}
A {\it combinatorial weight of type $((k),(m-k))$} is a $m$-tuple $\alpha=(\alpha_1,\ldots,\alpha_m)$ of symbols $\alpha_i\in\{\land,\lor\}$ with at least $k$ $\land$'s. Let $W^{((k),(m-k))}$ denote the set of all combinatorial weights of type $((k),(m-k))$.
\end{defi}

\begin{ex}\label{ex:comb_weights}
Here is a list of all combinatorial weights of type $((3),(1))$:
\[
\alpha=\land\land\land\land\,,\,\,\beta=\lor\land\land\land\,,\,\,\gamma=\land\lor\land\land\,,\,\,\delta=\land\land\lor\land\,,\,\,\epsilon=\land\land\land\lor.
\] 
Note that we omit the parentheses when writing down a combinatorial weight and we do not separate the symbols by commas. 
\end{ex} 

Given $\alpha\in W^{((k),(m-k))}$ and $l\in\{1,\ldots,m\}$, we define $F_l^\alpha=\spn(e_1,\ldots,e_i,f_1,\ldots,f_j)$, where $i$ is the number of $\land$'s amongst the first $l$ entries $\alpha_1,\ldots,\alpha_l$ of $\alpha$ and $j$ is the number of $\lor$'s amongst them. This yields a flag $(F^\alpha_1,\ldots,F^\alpha_m)\in Y_m$ associated to $\alpha\in W^{((k),(m-k))}$. Note that we have $e_k\in F_m^\alpha$ because by assumption there are at least $k$ $\land$'s in $\alpha$. In particular, $(F^\alpha_1,\ldots,F^\alpha_m)\in\mathcal Fl^{((k),(m-k)}$.

\begin{prop}\label{prop:fixed_points}
The map $\alpha\mapsto (F^\alpha_1,\ldots,F^\alpha_m)$ defines a bijection between $W^{((k),(m-k))}$ and the fixed points under the $\mathbb C^*$-action in $\mathcal Fl^{((k),(m-k))}$.
\end{prop}
\begin{proof}
The flag $(F^\alpha_1,\ldots,F^\alpha_m)\in Y_m$ is clearly a $\mathbb C^*$-fixed point for every $\alpha\in W^{((k),(m-k))}$ because each of its vector spaces is spanned by weight vectors with respect to the $\mathbb C^*$-action. Thus, we have a well-defined map which is injective by construction.

In order to see the surjectivity let $(F_1,\ldots,F_m)\in Y_m$ be a flag fixed by the $\mathbb C^*$-action. By induction on $l$ one can see that $F_l$ must be spanned by $e_1,\ldots,e_r,f_1,\ldots,f_s$ for some $r,s$ in order to be fixed by the $\mathbb C^*$-action. In particular, in doing so, one can simultaneously construct the preimage $\alpha$ of $(F_1,\ldots,F_m)$. Since $e_k\in F_m$, there must be at least $k$ $\land$'s in $\alpha$ and it follows that $\alpha\in W^{((k),(m-k))}$. 
\end{proof}    

\subsection{Attracting cells and their closures} In order to describe the attracting cell corresponding to a given fixed point, we first explain how to assign a (one-boundary) cup diagram to a combinatorial weight. Given a combinatorial weight $\alpha\in W^{((k),(m-k))}$, the associated cup diagram $C(\alpha)$ is obtained by successively connecting neighboring pairs of symbols $\lor\land$ by a cup (ignoring the symbols which are already paired). If there is no more $\lor$ to the left of a $\land$ we attach a ray to all remaining $\land$'s and a half-cup to all remaining $\lor$'s. 

\begin{ex} \label{ex:cup_diagrams_for_attracting_cells}
The cup diagrams associated with the combinatorial weights from Example~\ref{ex:comb_weights} are given by
\[
\begin{array}{ccccc}
C(\alpha) = 
\begin{tikzpicture}[scale=0.8,xscale=0.8,yscale=1,baseline={(0,0.5)}]
\draw[dotted] (0.5,0) -- (0.5,1) -- (4.5,1) -- (4.5,0) -- cycle;
\draw[thick] (1,1) -- +(0,-1);
\draw[thick] (2,1) -- +(0,-1);
\draw[thick] (3,1) -- +(0,-1);
\draw[thick] (4,1) -- +(0,-1);

\foreach \x in {1,2,3,4}{
\fill (\x,1) circle (2pt);}
\end{tikzpicture}

&,&
\hspace{.3em}C(\beta) = 

\begin{tikzpicture}[scale=0.8,xscale=0.8,yscale=1,baseline={(0,0.5)}]
\draw[dotted] (0.5,0) -- (0.5,1) -- (4.5,1) -- (4.5,0) -- cycle;
\draw[thick] (1,1) .. controls +(0,-.5) and +(0,-.5) .. (2,1);
\draw[thick] (3,1) -- +(0,-1);
\draw[thick] (4,1) -- +(0,-1);

\foreach \x in {1,2,3,4}{
\fill (\x,1) circle (2pt);}

\end{tikzpicture}

&,&
\hspace{.3em}C(\gamma) = 

\begin{tikzpicture}[scale=0.8,xscale=0.7,yscale=1,baseline={(0,0.5)}]
\draw[dotted] (0.5,0) -- (0.5,1) -- (4.5,1) -- (4.5,0) -- cycle;
\draw[thick] (1,1) -- +(0,-1); 
\draw[thick] (2,1) .. controls +(0,-.5) and +(0,-.5) .. (3,1);
\draw[thick] (4,1) -- +(0,-1);
\foreach \x in {1,2,3,4}{
\fill (\x,1) circle (2pt);}

\end{tikzpicture}

\end{array}
\]
\[
\begin{array}{cccc}
C(\delta) = 
\begin{tikzpicture}[scale=0.8,xscale=0.8,yscale=1,baseline={(0,0.5)}]
\draw[dotted] (0.5,0) -- (0.5,1) -- (4.5,1) -- (4.5,0) -- cycle;
\draw[thick] (1,1) -- +(0,-1);
\draw[thick] (2,1) -- +(0,-1); 
\draw[thick] (3,1) .. controls +(0,-.5) and +(0,-.5) .. (4,1);
\foreach \x in {1,2,3,4}{
\fill (\x,1) circle (2pt);}

\end{tikzpicture}
&,&
\hspace{.3em}C(\epsilon) = 
\begin{tikzpicture}[scale=0.8,xscale=0.8,yscale=1,baseline={(0,0.5)}]
\draw[dotted] (0.5,0) -- (0.5,1) -- (4.5,1) -- (4.5,0) -- cycle;
\draw[thick] (1,1) -- +(0,-1);
\draw[thick] (2,1) -- +(0,-1);
\draw[thick] (3,1) -- +(0,-1);
\draw[thick]  (4,1) to[out=270,in=180] (4.5,0.5); 

\foreach \x in {1,2,3,4}{
\fill (\x,1) circle (2pt);}

\end{tikzpicture}

\end{array}
\]
\end{ex}


Given a combinatorial weight $\alpha\in W^{((k),(m-k))}$, let 
\[
K_\alpha :=\big\{(F_1,\ldots,F_m)\in\mathcal Fl^{((k),(m-k))}\mid \lim_{t\to\infty}t\cdot (F_1,\ldots,F_m) = (F_1^\alpha,\ldots,F_m^\alpha)\big\}
\]
denote the attracting cell of the fixed point corresponding to $\alpha$ via Lemma~\ref{prop:fixed_points}.

Let $P=\spn(e_1,\ldots,e_N)$ and $Q=\spn(f_1,\ldots,f_N)$. Given a flag $(F_1,\ldots,F_m)\in\mathcal Fl^{((k),(m-k))}$, we associate with it a new flag $(F^\textrm{ass}_1,\ldots,F_m^\textrm{ass})\in\mathcal Fl^{((k),(m-k))}$ by setting $F_i^\textrm{ass}=P_i+Q_i\subseteq P\oplus Q=V$, where $P_i=F_i\cap P$ and $Q_i=\left(F_i+P\right)/P$ (we identify $V/P\cong Q$). Note that this flag is fixed under the $\mathbb C^*$-action which means we have $(F^\textrm{ass}_1,\ldots,F_m^\textrm{ass})=(F_1^\alpha,\ldots,F_m^\alpha)$ for some combinatorial weight $\alpha\in W^{((k),(m-k))}$. The next lemma explains how the flag $(F^\textrm{ass}_1,\ldots,F_m^\textrm{ass})$ associated with $(F_1,\ldots,F_m)$ can be used to check if $(F_1,\ldots,F_m)$ is contained in the attracting cell for $\alpha$.

\begin{lem}
A flag $(F_1,\ldots,F_m)$ in $\mathcal Fl^{((k),(m-k))}$ is contained in the attracting cell $K_\alpha$ corresponding to $\alpha\in W^{((k),(m-k))}$ if and only if $(F^\textrm{ass}_1,\ldots,F_m^\textrm{ass})=(F_1^\alpha,\ldots,F_m^\alpha)$.
\end{lem}
\begin{proof}
This follows exactly as in the proof of~\cite[Proposition 14]{SW12}.
\end{proof}

\begin{thm}\label{thm:bialynicki-birula}
The attracting cell $K_\alpha$ associated with the $\mathbb C^*$-fixed point $(F^\alpha_1,\ldots,F^\alpha_m)\in Y_m$ consists of all flags $(F_1,\ldots,F_m)\in\mathcal F^{((k),(m-k))}$ satisfying the following conditions imposed by the cup diagram $C(\alpha)$:
\[
\mathrm{(i)}\,\, F_j=z^{-\frac{1}{2}(j-i+1)}F_{i-1} \hspace{.8em} \text{ if } i \CupConnect j \hspace{1.9em} \mathrm{(ii)}\,\, F_i=F_{i-1}+\spn\left(e_{\frac{1}{2}(i+\rho_{C(\alpha)}(i))}\right) \hspace{.8em} \text{ if }\RayConnect 
\]
where $\rho_{C(\alpha)}(i)$ is as in Theorem~\ref{thm:main_result_1}. Additionally, we have the condition 
\[
\mathrm{(iii)}\,\,\, F_i\cap \spn(e_1,\ldots,e_N) = F_{i-1}\cap \spn(e_1,\ldots,e_N) 
\]
whenever vertex $i$ is a left endpoint of a cup or a half-cup.
\end{thm}

\begin{proof}
We prove the claim by induction on the number of cups in $C(\alpha)$. First assume that $\alpha$ is of the form $\land\land\ldots\land$ such that $C(\alpha)$ consists of rays only. Then $(F^\alpha_1,\ldots,F^\alpha_m)$ is given by $F_i(\alpha)=\spn(e_1,\ldots,e_i)$. For any other flag $(F_1,\ldots,F_m)$ attracted to $(F^\alpha_1,\ldots,F^\alpha_m)$ we have $P_i=F_i\cap P=\spn(e_1,\ldots,e_i)=F_i(\alpha)$ and thus $Q_i\cong 0$. In particular, it follows that any such $(F_1,\ldots,F_m)$ must already be equal to $(F^\alpha_1,\ldots,F^\alpha_m)$.

Next assume that $\alpha$ is of the form $\land\land\ldots\land\lor\ldots\lor$ such that $C(\alpha)$ consists of a sequence of rays followed by a sequence of half-cups (assume there exists at least one $\lor$). Let $\alpha'$ be the sequence obtained by deleting the rightmost $\lor$ in $\alpha$. By induction we have that $K_{\alpha'}=K^0_{\alpha'}$, where $K^0_{\alpha'}$ is defined as the set of flags satisfying the conditions (i) and (ii) of the theorem. Since $K_\alpha\subseteq\psi_m^{-1}(K_{\alpha'})$ we obtain
\[
K_\alpha\subseteq\psi_m^{-1}(K_{\alpha'})=\psi_m^{-1}(K^0_{\alpha'}),
\]
where $\psi_m^{-1}(K^0_{\alpha'})$ consists of all flags such that $F_1,\ldots,F_{m-1}$ satisfy the relations (i) and (ii) of $C(\alpha)$, and $F_m$ can be chosen freely. We even have $$K_\alpha\subseteq \{(F_1,\ldots,F_m)\in\psi_m^{-1}(K^0_{\alpha'})\mid F_i\cap P = F_{i-1}\cap P\}$$ because if $\alpha_i=\lor$, then we have $F_i\cap P=F_{i-1}\cap P$. 
On the other hand, if $P_m=F_m\cap P=F_{m-1}\cap P$ we see that $F_m^\textrm{ass}$ is obtained from $F_{m-1}^\textrm{ass}$ by adding a vector from $Q$. Thus, the inclusion is in fact an equality. 

Assume that there exists a cup in $C(\alpha)$ and assume that $i$ and $i+1$ are connected by a cup. In particular, $\alpha$ has symbols $\lor\land$ at positions $i$ and $i+1$. Let $\alpha'$ be the combinatorial weight obtained by removing these two symbols. By induction we have $K_{\alpha'}$ satisfies the conditions of the theorem and since $K_\alpha\subseteq\left(\psi_m^i\right)^{-1}(K_{\alpha'})$ (by the same argument as in~\cite{SW12}) we obtain $K_\alpha\subseteq\left(\psi_m^i\right)^{-1}(K_{\alpha'})=\left(\psi_m^i\right)^{-1}(K^0_{\alpha'})$. By arguing as in the case of a half-cup, we see that $K_\alpha=\{(F_1,\ldots,F_m)\in\left(\psi_m^i\right)^{-1}(K^0_{\alpha'})\mid F_i\cap P=F_{i-1}\cap P\}$.
\end{proof}

\begin{rem}
Instead of working with flags, we could also give a topological description of the attracting cells. Let $S_\alpha\subseteq (\mathbb S^2)^m$ denote the preimage of the attracting cell $K_\alpha$ under the homeomorphism from Theorem~\ref{thm:main_result_1}. Then $S_\alpha$ consists of all $(x_1,\ldots,x_m)\in(\mathbb S^2)^m$ which satisfy the conditions $x_i=-x_j$ and $x_i\neq p$ if $i<j$ are connected by a cup in $C(\alpha)$, $x_i\neq p$ if $i$ is connected to a half-cup, and $x_i=p$ if $i$ is connected to a ray in $C(\alpha)$. 
\end{rem}

\begin{ex}
The attracting cells corresponding to the diagrams in Example~\ref{ex:cup_diagrams_for_attracting_cells} are given by $K_\alpha=\big\{\left(\spn(e_1),\spn(e_1,e_2),\spn(e_1,e_2,e_3),\spn(e_1,e_2,e_3,e_4)\right)\big\}$ and
\[
K_\beta=\big\{\left(F_1,\spn(e_1,f_1),\spn(e_1,e_2,f_1),\spn(e_1,e_2,e_3,f_1)\right)\mid F_1\neq\spn(e_1)\big\},
\]
\[
K_\gamma=\big\{\left(\spn(e_1),F_2,\spn(e_1,e_2,f_1),\spn(e_1,e_2,e_3,f_1)\right)\mid F_2\neq\spn(e_1,e_2)\big\},
\]
\[
K_\delta=\big\{\left(\spn(e_1),\spn(e_1,e_2),F_3,\spn(e_1,e_2,e_3,f_1)\right)\mid F_3\neq\spn(e_1,e_2,e_3)\big\},
\]
\[
K_\epsilon=\big\{\left(\spn(e_1),\spn(e_1,e_2),\spn(e_1,e_2,e_3),F_4\right)\mid F_4\neq\spn(e_1,e_2,e_3,e_4)\big\}.
\]

Topologically, we have $S_\alpha=\big\{(p,p,p,p)\big\}$ and
\medskip

\begin{tabular}{ll}
$S_\beta=\{(x,-x,p,p)\mid x\in\mathbb S^2\setminus\{p\}\}$, & $S_\gamma=\{(p,x,-x,p)\mid x\in\mathbb S^2\setminus\{p\}\},$ \\
$S_\delta=\{(p,p,x,-x)\mid x\in\mathbb S^2\setminus\{p\}\}$, & $S_\epsilon=\{(p,p,p,x)\mid x\in\mathbb S^2\setminus\{p\}\}$.
\end{tabular}


\end{ex}
We would like to note two useful consequences.

\begin{cor}\label{cor:Bett_numbers}
The $2i$-th Betti number $b_i$, $0\leq i\leq m-k$, of the exotic Springer fiber $\mathcal Fl^{((k),(m-k))}$ is $b_{2i}=\binom{m}{i}$.
\end{cor}
\begin{proof}
This is a consequence of the paving constructed in Theorem~\ref{thm:bialynicki-birula}.
\end{proof}

\begin{cor}\label{cor:filtration}
The closure $\overline{K_\alpha}$ of the attracting cell $K_\alpha\subseteq Y_m$ is equal to $K_{C(\alpha)}\subseteq Y_m$. In particular, $\overline{K_\alpha}$ is homeomorphic to $S_{C(\alpha)}$. 
\end{cor}
\begin{proof}
The first claim is obtained by comparing the relations in Theorem~\ref{thm:bialynicki-birula} and Theorem~\ref{thm:algebraic_main_result}. Using this, the homeomorphism $\overline{K_\alpha}\cong S_{C(\alpha)}$ is then a consequence of Theorem~\ref{thm:main_result_1}. 
\end{proof}

\subsection{Intersections of closures of attracting cells}

The final goal of this section will be to give a combinatorial description of the pairwise intersections of the closures of attracting cells using oriented circle diagrams. We then describe a basis of the cohomology of these pairwise intersections in terms of oriented circle diagrams. This lays the foundation for constructing exotic versions of (generalized) arc algebras,~\cite{Kho02},~\cite{Str09},~\cite{BS11},~\cite{CK14}, by defining a convolution product on the direct sum of cohomologies of the pairwise intersections of the closures of the attracting cells (see~\cite{SW12} for the construction in type A).  

\begin{defi}
Let $\alpha,\beta\in W^{((k),(m-k))}$ and let $\alpha C(\beta)$ be the diagram obtained by sticking the weight $\alpha$ on top of the cup diagram $C(\beta)$. We say that $\alpha C(\beta)$ is an {\it oriented cup diagram} if $\alpha_i\neq \alpha_j$ for $i\CupConnect j$ and $\alpha_i=\land$ if $\RayConnect$. Furthermore, we allow any symbol $\alpha_i\in\{\land,\lor\}$ if $i$ is connected to a half-cup.    
\end{defi}

\begin{defi}
Given $\alpha,\beta\in W^{((k),(m-k))}$, an {\it oriented circle diagram} $\overline{C(\alpha)}\gamma C(\beta)$ is obtained by sticking a weight $\gamma\in W^{((k),(m-k))}$ between the two diagrams $\overline{C(\alpha)}$ and $C(\beta)$ such that both $\gamma C(\alpha)$ and $\gamma C(\beta)$ are both oriented cup diagrams. This induces orientations of the connected components, i.e., the circle and lines, of the circle diagram. We say that a component is oriented clockwise (resp.\ counterclockwise) if the leftmost vertex on such a component is a $\land$ (resp.\ $\lor$).
\end{defi}

The following theorem describes the pairwise intersections of the closures of the attracting cells in terms of oriented circle diagrams.
\begin{thm}\label{thm:intersection_of_attracting_cells}
Let $\alpha,\beta\in W^{((k),(m-k))}$. We have $\overline{K_\alpha}\cap\overline{K_\beta}\neq \emptyset$ if and only if there exists some $\gamma\in W^{((k),(m-k))}$ such that $\overline{C(\alpha)}\gamma C(\beta)$ is an oriented circle diagram.
\end{thm}
\begin{proof}
It is straightforward to check that there exists a $\gamma\in W^{((k),(m-k))}$ such that $\overline{C(\alpha)}\gamma C(\beta)$ is an oriented circle diagram if and only if all lines in $\overline{C(\alpha)}C(\beta)$ containing two rays are propagating. Since $\overline{K_\alpha}\cong S_{C(\alpha)}$ and $\overline{K_\beta}\cong S_{C(\beta)}$ by Corollary~\ref{cor:filtration}, the theorem now follows directly from Theorem~\ref{thm:intersection_of_components} and Remark~\ref{rem:generalization}.
\end{proof}

Let $\alpha,\beta\in W^{((k),(m-k))}$ be such that $\overline{K_\alpha}\cap\overline{K_\beta}\neq\emptyset$. Then $\overline{K_\alpha}\cap\overline{K_\beta}\cong S_{C(\alpha)}\cap S_{C(\beta)}\cong(\mathbb S^2)^k$ by Corollary~\ref{cor:filtration} and Theorem~\ref{thm:intersection_of_components}, where $k$ counts the number of circles and line segments containing two half-cups. In particular, we can find $2^k$ different $\gamma\in W^{((k),(m-k))}$ such that $\overline{C(\alpha)}\gamma C(\beta)$ is an oriented circle diagram because any circle or line segment containing two half-cups can be oriented either clockwise or counterclockwise. Since 
\[
H^*(\overline{K_\alpha}\cap\overline{K_\beta},\mathbb C)\cong H^*((\mathbb S^2)^k,\mathbb C)\cong \mathbb C[X]/(X^2)^{\otimes k},
\] 
we obtain an isomorphism of vector spaces
\[
H^*(\overline{K_\alpha}\cap\overline{K_\beta},\mathbb C)\cong\mathbb C[\overline{C(\alpha)}\gamma C(\beta)\mid \overline{C(\alpha)}\gamma C(\beta)\text{ is oriented}]
\]
sending an elementary tensor $a_1\otimes\ldots\otimes a_k$, $a_i\in\{1,X\}\subseteq\mathbb C[X]/(X^2)$, to the oriented circle diagram in which the components corresponding to spheres with $a_i=X$ are oriented clockwise.    

\section{The cohomology ring and an action of the Weyl group}\label{sec:cohomology_rings}


In this section we provide an explicit presentation of the cohomology ring of the exotic Springer fiber $\mathcal Fl^{((k),(m-k))}$. The main idea is to construct a topological space $\mathrm{Sk}_{2(m-k)}^m$ homotopy equivalent to $\mathcal Fl^{((k),(m-k))}$ whose cohomology ring will be straightforward to calculate. We then show that the type C Weyl group naturally acts on $\mathrm{Sk}_{2(m-k)}^m$ and induces Kato's exotic Springer representation in cohomology. Throughout this section we will write $H_*(X)$ (resp.\ $H^*(X)$) as a shorthand notation for $H_*(X,\mathbb Z)$ (resp.\ $H^*(X,\mathbb Z)$).

\subsection{Diagrammatic homology bases}

Our first goal is to construct diagrammatic homology bases of the submanifolds $S_\ba\subseteq\left(\mathbb S^2\right)^m$, $\ba\in\mathbb B^{((k),(m-k))}$, in the spirit of~\cite{RT11},~\cite{SW16}. We also provide a diagrammatic homology basis of $\left(\mathbb S^2\right)^m$. This allows us to explicitly describe the map induced by the natural inclusion $S_\ba\hookrightarrow \left(\mathbb S^2\right)^m$ in cohomology (see Lemma~\ref{lem:image_of_homology_gen}). We will use this result to compute the cohomology ring of the (topological) exotic Springer fiber $\mathcal S^{((k),(m-k))}\cong \mathcal Fl^{((k),(m-k))}$ in the next subsection.  

\begin{defi}
A {\it line diagram} is constructed by attaching $m$ vertical lines, each possibly decorated with a single dot, to $m$ vertices (numbered from left to right) on a horizontal line. Given $U\subseteq \{1,\ldots,m\}$, let $l_U$ denote the unique line diagram with dots precisely on the lines whose endpoints are not contained in $U$. We denote the set of all line diagrams on $m$ vertices by $\mathfrak{L}_m$. 
\end{defi}

\begin{ex}\label{ex:line_diagrams}
The line diagrams in $\mathfrak{L}_2$ are given as follows: 
\[
l_\emptyset=
\begin{tikzpicture}[baseline={(0,-.5)},scale=.8]
\draw[thick] (0,0) -- +(0,-1);
\draw[thick] (.5,0) -- +(0,-1);
\draw (0,-.5) circle(3pt);
\draw (.5,-.5) circle(3pt);
\end{tikzpicture}\hspace{1em},
\hspace{3em}
l_{\{1\}}=
\begin{tikzpicture}[baseline={(0,-.5)},scale=.8]
\draw[thick] (0,0) -- +(0,-1);
\draw[thick] (.5,0) -- +(0,-1);
\draw (.5,-.5) circle(3pt);
\end{tikzpicture}\hspace{1em},
\hspace{3em}
l_{\{2\}}=
\begin{tikzpicture}[baseline={(0,-.5)},scale=.8]
\draw[thick] (0,0) -- +(0,-1);
\draw[thick] (.5,0) -- +(0,-1);
\draw (0,-.5) circle(3pt);
\end{tikzpicture}\hspace{1em},
\hspace{3em}
l_{\{1,2\}}=
\begin{tikzpicture}[baseline={(0,-.5)},scale=.8]
\draw[thick] (0,0) -- +(0,-1);
\draw[thick] (.5,0) -- +(0,-1);
\end{tikzpicture}\,\,.
\]
\end{ex}

The two-sphere has a cell decomposition $\mathbb S^2=\{p\}\cup\left(\mathbb S^2\setminus\{p\}\right)$ consisting of a $0$-cell and a $2$-cell. Henceforth, we fix this CW-structure and equip $(\mathbb S^2)^m$ with the Cartesian product CW-structure. Note that we obtain a bijection between $\mathfrak{L}_m$ and the cells of the CW-complex $(\mathbb S^2)^m$ by mapping a line diagram $l_U$ to the cell $C_{l_U}$, where $C_{l_U}$ is defined by choosing the $0$-cell (resp.\ $2$-cell) for the $i$th sphere in $(\mathbb S^2)^m$ if the $i$th line of $l_U$ is dotted (resp.\ undotted). Since the homology classes $[C_{l_U}]$ of the cells form a basis of $H_*((\mathbb S^2)^m)$ as a $\mathbb Z$-module, we obtain an isomorphism of $\mathbb Z$-modules
\begin{equation} \label{eq:homology_identification_for_spheres}
\mathbb Z[\mathfrak{L}_m] \xrightarrow\cong H_*\left((\mathbb S^2)^m\right);\,\,\,l_U\mapsto [C_{l_U}],
\end{equation}
where $\mathbb Z[\mathfrak{L}_m]$ is the free $\mathbb Z$-module with basis $\mathfrak{L}_m$. The homological degree of a line diagram is given by twice the number of undotted lines.

\begin{ex}\label{ex:cells_spheres}
The cells of $\mathbb S^2\times\mathbb S^2$ corresponding to the line diagrams in $\mathfrak{L}_2$ (see Example~\ref{ex:line_diagrams}) are given by 

\smallskip
\begin{minipage}{.45\linewidth}
\begin{itemize}
\item $C_{l_\emptyset}=\{(p,p)\}$,
\item $C_{l_{\{1\}}}=\{(x,p)\mid x\in\mathbb S^2\setminus\{p\}\}$,
\end{itemize}
\end{minipage}
\begin{minipage}{.5\linewidth}
\begin{itemize}
\item $C_{l_{\{2\}}}=\{(p,x)\mid x\in\mathbb S^2\setminus\{p\}\}$,
\item $C_{l_{\{1,2\}}}=\{(x,y)\mid x,y\in\mathbb S^2\setminus\{p\}\}$.
\end{itemize}
\end{minipage}

\smallskip
\noindent The homology class $[C_{l_\emptyset}]$ (resp.\ $[C_{l_{\{1,2\}}}]$) is a basis of $H_0(\mathbb S^2\times\mathbb S^2)$ (resp.\ $H_4(\mathbb S^2\times\mathbb S^2)$) and $[C_{l_{\{1\}}}]$, $[C_{l_{\{2\}}}]$ form a basis of $H_2(\mathbb S^2\times\mathbb S^2)$.
\end{ex}

\begin{defi}
An {\it enriched cup diagram} is a cup diagram (as in Definition~\ref{defi:one-boundary_cup_diagrams}) in which the cups and half-cups are allowed to be decorated with a single dot each. Additionally, every ray is decorated with a single dot. Let $\widetilde{\mathbb B}^{((k),(m-k))}$ be the set of all enriched cup diagrams on $m$ vertices such that the total number of cups plus half-cups is $m-k$.
\end{defi}

\begin{ex} \label{ex:enriched_cup_diagrams}
Here are two examples of enriched cup diagrams in $\widetilde{\mathbb B}^{((4),(5))}$:

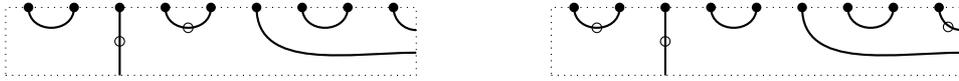
\begin{multicols}{2}
\begin{tikzpicture}[scale=0.6,xscale=1,baseline={(0,-.45)}]
\draw[dotted] (0.5,0) -- (9.5,0) -- (9.5,-1.5) -- (0.5,-1.5) -- cycle;

\foreach \x in {1,2,...,9}
\fill (\x,0) circle (3pt);

\draw (4.5,-0.45) circle (3pt);
\draw (3,-0.75) circle (3pt);

\draw[thick] (1,0) .. controls +(0,-0.6) and +(0,-0.6) .. +(1,0);
\draw[thick] (4,0) .. controls +(0,-0.6) and +(0,-0.6) .. +(1,0);
\draw[thick] (7,0) .. controls +(0,-0.6) and +(0,-0.6) .. +(1,0);

\draw[thick] (3,0) -- +(0,-1.5);
\draw[thick] (6,0) to [out=270,in=180](9.5,-1);
\draw[thick] (9,0) to [out=270,in=180](9.5,-0.5);
\end{tikzpicture}

\columnbreak

\begin{tikzpicture}[scale=0.6,xscale=1,baseline={(0,-.45)}]
\draw[dotted] (0.5,0) -- (9.5,0) -- (9.5,-1.5) -- (0.5,-1.5) -- cycle;

\foreach \x in {1,2,...,9}
\fill (\x,0) circle (3pt);

\draw (1.5,-0.45) circle (3pt);
\draw (3,-0.75) circle (3pt);
\draw (9.2,-0.43) circle (3pt);

\draw[thick] (1,0) .. controls +(0,-0.6) and +(0,-0.6) .. +(1,0);
\draw[thick] (4,0) .. controls +(0,-0.6) and +(0,-0.6) .. +(1,0);
\draw[thick] (7,0) .. controls +(0,-0.6) and +(0,-0.6) .. +(1,0);

\draw[thick] (3,0) -- +(0,-1.5);
\draw[thick] (6,0) to [out=270,in=180](9.5,-1);
\draw[thick] (9,0) to [out=270,in=180](9.5,-0.5);

\end{tikzpicture}
\end{multicols}
\end{ex}

Given $\ba\in\mathbb B^{((k),(m-k))}$, let $i_1<i_2<\cdots<i_{m-k}$ denote the left endpoints of the cups and half-cups in $\ba$. Note that we have a homeomorphism
\begin{equation}\label{eq:homeo_giving_cell_decomp}
\xi_\ba\colon S_\ba \xrightarrow\cong (\mathbb S^2)^{m-k};\,\,(x_1,\ldots,x_m) \mapsto (x_{i_1},\ldots,x_{i_{m-k}}),
\end{equation}
and we equip $S_\ba$ with the structure of the CW-complex whose cells are defined as the preimages of the cells of $(\mathbb S^2)^{m-k}$  under the homeomorphism (\ref{eq:homeo_giving_cell_decomp}). We obtain a bijection between $\widetilde{\mathbb B}^{((k),(m-k))}$ and the cells of $S_\ba$ by mapping an enriched cup diagram $M$ to the cell $C_M$, where a dotted cup (resp.\ undotted cup) means that we have chosen the $0$-cell (resp.\ $2$-cell) for the corresponding sphere via (\ref{eq:homeo_giving_cell_decomp}). Since the homology classes of the the cells of $S_\ba$ form a basis of the homology $H_*(S_\ba)$, we obtain an isomorphism of $\mathbb Z$-modules
\begin{equation} \label{eq:homology_identification_for_S_ba}
\mathbb Z[\widetilde{\mathbb B}^{((k),(m-k))}] \xrightarrow\cong H_*(S_\ba);\,\, M\mapsto [C_M]. 
\end{equation}
An enriched cup diagram for which the number of undotted cups plus undotted half-cups is equal to $l$ has homological degree $2l$.


\begin{defi}
Given an enriched cup diagram $M\in\widetilde{\mathbb B}^{((k),(m-k))}$, we define the associated {\it line diagram sum} as 
\[ 
L_M=\sum_{U\in\mathcal U_M}(-1)^{\Lambda_M(U)}l_U \in \mathbb{Z}[\mathfrak{L}_m],
\]
where $\mathcal U_M$ is the set of all subsets $U\subseteq\{1,\ldots,m\}$ containing precisely one endpoint of every undotted cup as well as the endpoints of every undotted half-cup in $M$. Moreover, $\Lambda_M(U)$ counts the total number of right endpoints of cups in $U$.
\end{defi}

\begin{ex}
Consider the enriched cup diagram
\[
M=
\begin{tikzpicture}[scale=0.6,xscale=1,baseline={(0,-.45)}]
\draw[dotted] (0.5,0) -- (9.5,0) -- (9.5,-1.5) -- (0.5,-1.5) -- cycle;

\foreach \x in {1,2,...,9}
\node[above] at (\x,0) {\tiny{$\x$}};
\foreach \x in {1,2,...,9}
\fill (\x,0) circle (3pt);

\draw (1.5,-0.45) circle (3pt);
\draw (3,-0.75) circle (3pt);
\draw (9.2,-0.43) circle (3pt);

\draw[thick] (1,0) .. controls +(0,-0.6) and +(0,-0.6) .. +(1,0);
\draw[thick] (4,0) .. controls +(0,-0.6) and +(0,-0.6) .. +(1,0);
\draw[thick] (7,0) .. controls +(0,-0.6) and +(0,-0.6) .. +(1,0);

\draw[thick] (3,0) -- +(0,-1.5);
\draw[thick] (6,0) to [out=270,in=180](9.5,-1);
\draw[thick] (9,0) to [out=270,in=180](9.5,-0.5);

\end{tikzpicture}
\in \widetilde{\mathbb{B}}^{((4),(5))}.
\]
We have
\[
\mathcal{U}_M=\{\{4,6,7\}, \{4,6,8\}, \{5,6,7\}, \{5,6,8\} \}
\]
with 
\[
\Lambda_{M}(\{4,6,7\})=0, \Lambda_{M}(\{4,6,8\})=\Lambda_{M}(\{5,6,7\})=1 \, \,  \text{and} \, \,  \Lambda_{M}(\{5,6,8\})=2.
\]
Therefore, the associated line diagram sum is 
\[
L_M=l_{\{4,6,7\}} - l_{\{4,6,8\}} -  l_{\{5,6,7\}} +  l_{\{5,6,8\}}\in\mathbb{Z}[\mathfrak{L}_9].
\]
\end{ex}

\begin{lem} \label{lem:image_of_homology_gen}
The map $(\psi_\ba)_*\colon H_*(S_\ba)\to H_*((\mathbb S^2)^m)$ induced by the natural inclusion $$\psi_\ba\colon S_\ba\hookrightarrow (\mathbb S^2)^m$$ is explicitly given by the assignment $M\mapsto L_M$ (via identifications (\ref{eq:homology_identification_for_spheres}) and (\ref{eq:homology_identification_for_S_ba})).
\end{lem}
\begin{proof}
We first consider a cup diagram $\ba_0\in\mathbb B^{((k),(m-k))}$ in which all $l$ half-cups, $0\leq l\leq m-k$, are connected to the rightmost vertices $\{m-l+1,\ldots,m\}$. Let $\ba_0^\prime$ be the cup diagram on $m-l$ vertices obtained by deleting the $l$ rightmost vertices in $\ba_0$ and let $\ba_0^{\prime\prime}$ be the cup diagram on $l$ vertices obtained by deleting the leftmost $m-l$ vertices in $\ba_0$. Then $\ba_0^\prime$ consists of cups and rays only and $\ba_0^{\prime\prime}$ consists of half-cups only. Let $M_0$ be an enriched cup diagram obtained by placing dots on $\ba_0$ and let $M^\prime_0$ (resp.\ $M^{\prime\prime}_0$) be the corresponding diagram for $\ba_0^\prime$ (resp.\ $\ba_0^{\prime\prime}$).

By~\cite[\S 2]{Rus11}, the manifold $S_{\ba_0^\prime}$ also appears as a component of a topological Springer fiber of type A and it follows from~\cite[Lemma 3.12]{RT11} and~\cite[\S 5]{Rus11} (see also~\cite[Proposition 12]{SW16}) that the map $H_*(S_{\ba_0^\prime})\rightarrow H_*((\mathbb S^2)^{m-l})$ induced by the natural inclusion $\psi_{\ba_0^\prime}$ is given by the assignment $M_0^\prime\mapsto L_{M_0^\prime}$ as claimed. 

Next, note that we can identify $S_{\ba_0^{\prime\prime}}$ with $(\mathbb S^2)^l$ and the inclusion $\psi_{\ba_0^{\prime\prime}}$ with the identity $\mathrm{id}_{(\mathbb S^2)^l}$. In particular, the map $(\psi_{\ba_0^{\prime\prime}})_*$ is evidently given by $M_0^{\prime\prime}\mapsto l_V$, where $V\subseteq\{1,\ldots,l\}$ contains all vertices connected to the undotted half-cups in $M_0^{\prime\prime}$.   

Thus, by the preceding two paragraphs, we can compute the image of $M_0^\prime \otimes M_0^{\prime\prime}$ under the map 
\begin{equation}\label{eq:tensor_product_map}
(\psi_{\ba_0^\prime})_*\otimes(\psi_{\ba_0^{\prime\prime}})_*\colon H_*\left(S_{\ba_0^\prime}\right)\otimes H_*\left(S_{\ba_0^{\prime\prime}}\right)\to H_*\left((\mathbb S^2)^{m-l}\right)\otimes H_*\left((\mathbb S^2)^l\right)
\end{equation} 
as follows:
\begin{align*}
(\psi_{\ba_0^\prime})_*\otimes(\psi_{\ba_0^{\prime\prime}})_*(M_0^\prime \otimes M_0^{\prime\prime}) = L_{M_0^\prime}\otimes l_V &= \left(\sum_{U^\prime\in\mathcal U_{M_0^\prime}}(-1)^{\Lambda_{M_0^\prime}(U^\prime)}l_{U^\prime}\right)\otimes l_V\\
&=\sum_{U^\prime\in\mathcal U_{M_0^\prime}}(-1)^{\Lambda_{M_0^\prime}(U^\prime)}\left(l_{U^\prime}\otimes l_V\right).
\end{align*}

Note that the K\"unneth isomorphism 
\begin{equation}\label{eq:cross-product_1}
H_*\left(S_{\ba_0^\prime}\right)\otimes H_*\left(S_{\ba_0^{\prime\prime}}\right)\cong H_*\left(S_{\ba_0^\prime}\times S_{\ba_0^{\prime\prime}}\right)=H_*\left(S_{\ba_0}\right)
\end{equation}
sends $M_0^\prime\otimes M_0^{\prime\prime}$ to $M_0$ and 
\begin{equation}\label{eq:cross-product_2}
H_*\left((\mathbb S^2)^{m-l}\right)\otimes H_*\left((\mathbb S^2)^l\right)\cong H_*\left((\mathbb S^2)^{m-l}\times (\mathbb S^2)^l\right)=H_*\left((\mathbb S^2)^m\right)
\end{equation}
sends $l_{U^\prime}\otimes l_V$ to $l_U$, where $U$ contains all numbers in $U^\prime$ and all numbers obtained by adding $m-l$ to the numbers in $V$. Since the map $(\psi_{\ba_0})_*=(\psi_{\ba_0^\prime}\times\psi_{\ba_0^{\prime\prime}})_*$ is given by the composition of the maps in (\ref{eq:tensor_product_map}),(\ref{eq:cross-product_1}) and (\ref{eq:cross-product_2}), it is straightforward to check that it maps $M_0$ to $L_{M_0}$ as claimed (note that $\Lambda_{M_0^\prime}(U^\prime)=\Lambda_{M_0}(U)$).

For the general case let $\ba\in\mathbb B^{((k),(m-k))}$ be a cup diagram and let $M$ be an enriched cup diagram obtained by decorating $\ba$ with dots. Let $\ba_0\in\mathbb B^{((k),(m-k))}$ be the cup diagram obtained by rearranging the components of $\ba$ in such a way that the half-cups are all connected to the rightmost vertices and let $M_0$ be the enriched cup diagram for $\ba_0$ in which a component has a dot if and only if the corresponding component of $M$ has a dot. Define $\tau_\ba$ to be the permutation of the vertices realizing this rearrangement. We obtain an induced homeomorphism $\tau_\ba\colon(\mathbb S^2)^m\to(\mathbb S^2)^m$ which permutes the coordinates and restricts to a homeomorphism $S_\ba\to S_{\ba_0}$. We can then write the inclusion $\psi_\ba\colon S_\ba\hookrightarrow(\mathbb S^2)^m$ as the composition $\tau_\ba\circ\psi_{\ba_0}\circ\left(\tau^{-1}_\ba\vert_{S_\ba}\right)$. Then the claim in Lemma~\ref{lem:image_of_homology_gen} follows from the following calculation:
\begin{align*}
(\psi_\ba)_*(M) &=(\tau_\ba)_*\circ(\psi_{\ba_0})_*\circ(\tau^{-1}_\ba\vert_{S_\ba})_*(M)=(\tau_\ba)_*\circ(\psi_{\ba_0})_*(M_0)\\
&=(\tau_\ba)_*\left(\sum_{U\in\mathcal U_{M_0}}(-1)^{\Lambda_{M_0}(U)}l_U\right)=\sum_{U\in\mathcal U_M}(-1)^{\Lambda_M(U)}l_U.\qedhere 
\end{align*}
\end{proof}

\subsection{Ring structure of cohomology} In this subsection we construct a topological space which is homotopy equivalent to the exotic Springer fiber $\mathcal Fl^{((k),(m-k))}$ (see Theorem~\ref{thm:cohomology_ring}\ref{thm:cohomology_part_a}). Computing the cohomology ring of this space turns out to be straightforward. Thus, we obtain an explicit presentation of the cohomology ring of $\mathcal Fl^{((k),(m-k))}$ (see Theorem~\ref{thm:cohomology_ring}\ref{thm:cohomology_part_b}). 

For technical purposes we introduce a map
\[
\beta^{((k),(m-k))}\colon\bigsqcup_{l=0}^{m-k}\mathbb B^{((m-l),(l))}\longrightarrow\widetilde{\mathbb B}^{((k),(m-k))}.
\]
In order to define this map let $\ba\in\mathbb B^{((m-l),(l))}$ be a cup diagram, $0\leq l\leq m-k$. Let $T_{\ba}$ be the standard one-row bitableau of shape $((m-l),(l))$ associated to $\ba$ via Lemma~\ref{lem:bijection_tableaux_cups} and let $\widetilde{T}_{\ba}$ be the unique standard one-row bitableau of shape $((k),(m-k))$ obtained by moving the $m-l-k$ largest entries in the left tableau of $T_{\ba}$ to the right tableau (and reordering the right tableau to make it standard). The tableau $\widetilde{T}_{\ba}$ gives rise to a cup diagram in $\mathbb B^{((m-k),(k))}$ (again, via Lemma~\ref{lem:bijection_tableaux_cups}) which we additionally decorate by putting a dot on all rays as well as on each component which is connected to a vertex labeled by one of the $m-l-k$ numbers which we moved. We define the resulting enriched cup diagram in $\widetilde{\mathbb B}^{((k),(m-k))}$ as the image of $\ba$ under $\beta^{((k),(m-k))}$. 

\begin{ex}
Let $m=5$ and $k=3$. In the following, we explicitly describe the map
\[
\beta^{((3),(2))}\colon\mathbb{B}^{((5),(-))} \sqcup \mathbb{B}^{((4),(1))} \sqcup \mathbb{B}^{((3),(2))} \longrightarrow \widetilde{\mathbb B}^{((3),(2))}.
\]
\begin{enumerate}
\item Let $\textbf{a}=
\begin{tikzpicture}[scale=0.5,xscale=1,baseline={(0,-0.5)}]
\draw[dotted] (0.5,0) -- (5.5,0) -- (5.5,-2.2) -- (0.5,-2.2) -- cycle;

\foreach \x in {1,2,3,4,5}
\node[above] at (\x,0) {\tiny{$\x$}};
\foreach \x in {1,2,3,4,5}
\fill (\x,0) circle (3pt);

\draw[thick] (1,0) -- +(0,-2.2);
\draw[thick] (2,0) -- +(0,-2.2);
\draw[thick] (3,0) -- +(0,-2.2);
\draw[thick] (4,0) -- +(0,-2.2);
\draw[thick] (5,0) -- +(0,-2.2);
\end{tikzpicture}
\in \mathbb{B}^{((5),(-))}. 
$
The bitableau associated to $\ba$ is $$T_{\ba}=(\young(12345),-).$$ Since $l=0$ in this case, we have $m-l-k=2$ and so $\widetilde{T}_{\ba}=(\young(123),\young(45))$, which yields the following image:
\[
\beta^{((3),(2))}(\ba)=
\begin{tikzpicture}[scale=0.5,xscale=1,baseline={(0,-0.5)}]
\draw[dotted] (0.5,0) -- (5.5,0) -- (5.5,-2.2) -- (0.5,-2.2) -- cycle;

\foreach \x in {1,2,3,4,5}
\node[above] at (\x,0) {\tiny{$\x$}};
\foreach \x in {1,2,3,4,5}
\fill (\x,0) circle (3pt);

\draw (5,-0.5) circle (4pt);
\draw (3,-0.75) circle (4pt);
\draw (4.2,-0.75) circle (4pt);
\draw (1,-0.75) circle (4pt);
\draw (2,-0.75) circle (4pt);

\draw[thick] (1,0) -- +(0,-2.2);
\draw[thick] (2,0) -- +(0,-2.2);
\draw[thick] (3,0) -- +(0,-2.2);
\draw[thick] (4,0) to [out=270,in=180](5.5,-1.6);
\draw[thick] (5,0) to [out=270,in=180](5.5,-1.2);
\end{tikzpicture}
\in \widetilde{\mathbb{B}}^{((3),(2))}. 
\]

\item Now let $\bb=
\begin{tikzpicture}[scale=0.5,xscale=1,baseline={(0,-0.5)}]
\draw[dotted] (0.5,0) -- (5.5,0) -- (5.5,-2.2) -- (0.5,-2.2) -- cycle;

\foreach \x in {1,2,3,4,5}
\node[above] at (\x,0) {\tiny{$\x$}};
\foreach \x in {1,2,3,4,5}
\fill (\x,0) circle (3pt);

\draw[thick] (1,0) -- +(0,-2.2);
\draw[thick] (2,0) -- +(0,-2.2);
\draw[thick] (5,0) -- +(0,-2.2);

\draw[thick] (3,0) .. controls +(0,-1) and +(0,-1) .. +(1,0);
\end{tikzpicture}
\in \mathbb{B}^{((4),(1))}$. The bitableau associated to $\bb$ is $$T_{\bb}=(\young(1245),\young(3)).$$ In this case $l=1$ and so $m-k-l=1$ giving $\widetilde{T}_{\bb}=(\young(124),\young(35))$. Therefore:
\[
\beta^{((3),(2))}(\bb)=
\begin{tikzpicture}[scale=0.5,xscale=1,baseline={(0,-0.5)}]
\draw[dotted] (0.5,0) -- (5.5,0) -- (5.5,-1.7) -- (0.5,-1.7) -- cycle;

\foreach \x in {1,2,3,4,5}
\node[above] at (\x,0) {\tiny{$\x$}};
\foreach \x in {1,2,3,4,5}
\fill (\x,0) circle (3pt);

\draw (1,-0.75) circle (4pt);
\draw (2,-0.75) circle (4pt);
\draw (5.1,-0.75) circle (4pt);

\draw[thick] (1,0) -- +(0,-1.7);
\draw[thick] (2,0) -- +(0,-1.7);
\draw[thick] (5,0) to [out=270,in=180](5.5,-1);

\draw[thick] (3,0) .. controls +(0,-1) and +(0,-1) .. +(1,0);
\end{tikzpicture}
\in \widetilde{\mathbb{B}}^{((3),(2))}. 
\]
\item As a final example, let 
$\textbf{c}=
\begin{tikzpicture}[scale=0.5,xscale=1,baseline={(0,-0.5)}]
\draw[dotted] (0.5,0) -- (5.5,0) -- (5.5,-2.2) -- (0.5,-2.2) -- cycle;

\foreach \x in {1,2,3,4,5}
\node[above] at (\x,0) {\tiny{$\x$}};
\foreach \x in {1,2,3,4,5}
\fill (\x,0) circle (3pt);

\draw[thick] (1,0) .. controls +(0,-1.5) and +(0,-1.5) .. +(3,0);
\draw[thick] (2,0) .. controls +(0,-1) and +(0,-1) .. +(1,0);
\draw[thick] (5,0) -- +(0,-2.2);
\end{tikzpicture}
\in \mathbb{B}^{((3),(2))}. 
$
In this case we have $$T_{\textbf{c}}=\widetilde{T}_{\textbf{c}}=(\young(345),\young(12))$$ since $l=2$. Therefore we have, 
\[
\beta^{((3),(2))}(\textbf{c})=
\begin{tikzpicture}[scale=0.5,xscale=1,baseline={(0,-0.5)}]
\draw[dotted] (0.5,0) -- (5.5,0) -- (5.5,-2.2) -- (0.5,-2.2) -- cycle;

\foreach \x in {1,2,3,4,5}
\node[above] at (\x,0) {\tiny{$\x$}};
\foreach \x in {1,2,3,4,5}
\fill (\x,0) circle (3pt);

\draw (5,-1) circle (4pt);
\draw[thick] (1,0) .. controls +(0,-1.5) and +(0,-1.5) .. +(3,0);
\draw[thick] (2,0) .. controls +(0,-1) and +(0,-1) .. +(1,0);
\draw[thick] (5,0) -- +(0,-2.2);
\end{tikzpicture}
\in \widetilde{\mathbb{B}}^{((3),(2))}. 
\]
\end{enumerate}
\end{ex}

\begin{defi}
An enriched cup diagram in $\widetilde{\mathbb B}^{((k),(m-k))}$ contained in the image of $\beta^{((k),(m-k))}$ is called a {\it standard enriched cup diagram}. We write $s\widetilde{\mathbb B}^{((k),(m-k))}$ to denote the set of all standard enriched cup diagrams on $m$ vertices such that the total number of cups and half-cups is $m-k$. 
\end{defi}

\begin{ex}
The set $s\widetilde{\mathbb B}^{((3),(1))}$ consists of the following five diagrams:
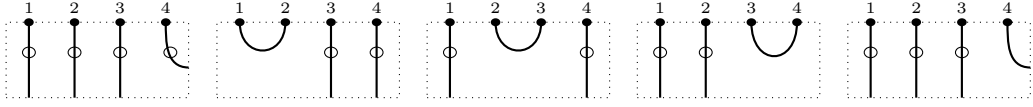
\[
\begin{tikzpicture}[scale=0.5,xscale=1.2,baseline={(0,-0.5)}]
\draw[dotted] (0.5,0) -- (4.5,0) -- (4.5,-2) -- (0.5,-2) -- cycle;
\foreach \x in {1,2,3,4}
\node[above] at (\x,0) {\tiny{$\x$}};
\foreach \x in {1,2,3,4}
\fill (\x,0) circle (3pt);

\draw (1,-0.8) circle (4pt);
\draw (2,-0.8) circle (4pt);
\draw (3,-0.8) circle (4pt);
\draw (4.1,-0.8) circle (4pt);

\draw[thick] (1,0) -- +(0,-2);
\draw[thick] (2,0) -- +(0,-2);
\draw[thick] (3,0) -- +(0,-2);
\draw[thick] (4,0) to [out=270, in=180] (4.5,-1.2);
\end{tikzpicture}
\quad 
\begin{tikzpicture}[scale=0.5,xscale=1.2,baseline={(0,-0.5)}]
\draw[dotted] (0.5,0) -- (4.5,0) -- (4.5,-2) -- (0.5,-2) -- cycle;
\foreach \x in {1,2,3,4}
\node[above] at (\x,0) {\tiny{$\x$}};
\foreach \x in {1,2,3,4}
\fill (\x,0) circle (3pt);
\draw (3,-0.8) circle (4pt);
\draw (4,-0.8) circle (4pt);

\draw[thick] (1,0) .. controls +(0,-1) and +(0,-1) .. +(1,0);

\draw[thick] (3,0) -- +(0,-2);
\draw[thick] (4,0) -- +(0,-2);
\end{tikzpicture}
\quad
\begin{tikzpicture}[scale=0.5,xscale=1.2,baseline={(0,-0.5)}]
\draw[dotted] (0.5,0) -- (4.5,0) -- (4.5,-2) -- (0.5,-2) -- cycle;
\foreach \x in {1,2,3,4}
\node[above] at (\x,0) {\tiny{$\x$}};
\foreach \x in {1,2,3,4}
\fill (\x,0) circle (3pt);

\draw (1,-0.8) circle (4pt);
\draw (4,-0.8) circle (4pt);

\draw[thick] (2,0) .. controls +(0,-1) and +(0,-1) .. +(1,0);
\draw[thick] (1,0) -- +(0,-2);
\draw[thick] (4,0) -- +(0,-2);
\end{tikzpicture}
\quad
\begin{tikzpicture}[scale=0.5,xscale=1.2,baseline={(0,-0.5)}]
\draw[dotted] (0.5,0) -- (4.5,0) -- (4.5,-2) -- (0.5,-2) -- cycle;
\foreach \x in {1,2,3,4}
\node[above] at (\x,0) {\tiny{$\x$}};
\foreach \x in {1,2,3,4}
\fill (\x,0) circle (3pt);

\draw (1,-0.8) circle (4pt);
\draw (2,-0.8) circle (4pt);

\draw[thick] (3,0) .. controls +(0,-1.2) and +(0,-1.2) .. +(1,0);
\draw[thick] (1,0) -- +(0,-2);
\draw[thick] (2,0) -- +(0,-2);
\end{tikzpicture}
\quad
\begin{tikzpicture}[scale=0.5,xscale=1.2,baseline={(0,-0.5)}]
\draw[dotted] (0.5,0) -- (4.5,0) -- (4.5,-2) -- (0.5,-2) -- cycle;
\foreach \x in {1,2,3,4}
\node[above] at (\x,0) {\tiny{$\x$}};
\foreach \x in {1,2,3,4}
\fill (\x,0) circle (3pt);

\draw (1,-0.8) circle (4pt);
\draw (2,-0.8) circle (4pt);
\draw (3,-0.8) circle (4pt);

\draw[thick] (1,0) -- +(0,-2);
\draw[thick] (2,0) -- +(0,-2);
\draw[thick] (3,0) -- +(0,-2);
\draw[thick] (4,0) to [out=270, in=180] (4.5,-1.2);
\end{tikzpicture}
\]
The first diagram is the image of $\mathbb{B}^{((4),(-))}$ under the map $\beta^{((3),(1))}$ and the other diagrams constitute the image of the set $\mathbb{B}^{((3),(1))}$.  
\end{ex}

\begin{rem}\label{rem:standard_enriched_cup_diagram_remark}
The left endpoints of the cups and half-cups in $\ba\in\mathbb B^{((m-l),(l))}$ are precisely the left endpoints of the undotted cups and undotted half-cups in $\beta^{((k),(m-k))}(\ba)$. Thus, $\beta^{((k),(m-k))}$ maps $\mathbb B^{((m-l),(l))}$ to the enriched cup diagrams of degree $2l$. Moreover, since the sets containing all vertices connected to the left endpoints of cups and half-cups are different for $\ba,\bb\in\mathbb B^{((m-l),(l))}$ if $\ba\neq\bb$ (because these vertices are precisely the entries of the right tableau in the bitableau associated to a cup diagram via the bijection in Lemma~\ref{lem:bijection_tableaux_cups}), it follows that the standard enriched cup diagrams associated to $\ba$ and $\bb$ via $\beta^{((k),(m-k))}$ must be different. In other words, a standard enriched cup diagram is uniquely determined by the left endpoints of its undotted cups and undotted half-cups.  
\end{rem}


\begin{lem}\label{lem:linear_independence}
The line diagram sums $L_M$, where $M$ varies over all standard enriched cup diagrams in $s\widetilde{\mathbb B}^{((k),(m-k))}$, are $\mathbb Z$-linearly independent in $\mathbb Z[\mathfrak{L}_m]$.
\end{lem}
\begin{proof}
Note that it suffices to prove that the elements $L_M$, where $M$ varies over all standard enriched cup diagrams in $s\widetilde{\mathbb B}^{((k),(m-k))}$ such that the total number of undotted cups and undotted half-cups is precisely $l$, $0\leq l\leq m$, are $\mathbb Z$-linearly independent. In the following we write $s\widetilde{\mathbb B}^{((k),(m-k))}_l$ to denote the set of all such standard enriched cup diagrams.   

We begin by defining a total order on the subsets of $\{1,\ldots,m\}$ of cardinality $l$ by setting
\begin{equation} \label{eq:total_order}
\{i_1<\cdots<i_l\}<\{i'_1<\cdots<i'_l\} \, \Leftrightarrow \, \exists\, r \, \colon i_r<i_r' \text{ and }i_1=i'_1,\ldots,i_{r-1}=i'_{r-1}.
\end{equation}

This induces a total order on all line diagrams $l_U\in\mathfrak{L}_m$, where $U\subseteq\{1,\ldots,m\}$ has cardinality $l$, i.e., we obtain a total order on our basis of the submodule of $\mathbb Z[\mathfrak{L}_m]$ spanned by all line diagrams of degree $2l$.

Given a standard enriched cup diagram $M\in s\widetilde{\mathbb B}^{((k),(m-k))}_l$, we define $U_M\in\mathcal U_M$ as the set containing the left endpoints of all undotted cups as well as the endpoints of all undotted half-cups. Given a different standard enriched cup diagrams $N\in s\widetilde{\mathbb B}^{((k),(m-k))}_l\setminus\{M\}$, we have $U_M\neq U_N$ because knowing the left endpoints of undotted cups and undotted half-cups uniquely determines a standard enriched cup diagram in $s\widetilde{\mathbb B}^{((k),(m-k))}_l$ by Remark~\ref{rem:standard_enriched_cup_diagram_remark}. Hence, it makes sense to define a total order on $s\widetilde{\mathbb B}^{((k),(m-k))}_l$ by setting $M < N$ if and only if $U_M < U_N$ with respect to the order (\ref{eq:total_order}).

Let $M_1<M_2<\ldots<M_{\binom{m}{l}}$ denote the elements of $s\widetilde{\mathbb B}^{((k),(m-k))}_l$ and assume $\sum_i\lambda_iL_{M_i}=0$ for $\lambda_i\in\mathbb Z$. Then the line diagram $l_{U_{M_1}}$ appears in $L_{M_1}$ with non-zero coefficient but it does not appear in any other $L_{M_i}$, $i\neq 1$, because we have $l_{U_{M_1}}<l_{U_{M_i}}$ and $l_{U_{M_i}}<l_U$ for all $U\in\mathcal U_{M_i}\setminus\{U_{M_i}\}$. Hence, we deduce that $\lambda_1=0$. By repeating the same argument with $M_2<\ldots<M_{\binom{m}{l}}$ we obtain that $\lambda_2=0$ and continuing in this way shows that all coefficients must be zero. 
\end{proof}

\begin{prop}\label{prop:injection}
The natural inclusion of $\mathcal S^{((k),(m-k))}$ into $(\mathbb S^2)^m$ induces an injection $H_*(\mathcal S^{((k),(m-k))})\hookrightarrow H_*((\mathbb S^2)^m)$ in homology. We even have an isomorphism 
\[
H_{2l}(\mathcal S^{((k),(m-k))})\cong H_{2l}((\mathbb S^2)^m)
\] 
of $\mathbb Z$-modules for $0\leq l\leq m-k$.
\end{prop}
\begin{proof}
Given $\ba\in\mathbb B^{((k),(m-k))}$, let $\phi_\ba\colon S_\ba\hookrightarrow\mathcal S^{((k),(m-k))}$, $\psi_\ba\colon S_\ba\hookrightarrow (\mathbb S^2)^m$ be the natural inclusions. We obtain a commutative diagram
\begin{equation} \label{eq:key_comm_diagram}
\begin{xy}
	\xymatrix{
		\displaystyle\bigoplus\limits_{\ba\in\mathbb B^{((k),(m-k))}} H_*(S_\ba) \ar[rr]^{\phi_{((k),(m-k))}} \ar@/^1.7pc/[rrrr]^{\psi_{((k),(m-k))}} && H_*(\mathcal S^{((k),(m-k))}) \ar[rr]^{\gamma_{((k),(m-k))}}	&& H_*((\mathbb S^2)^m)
	}
\end{xy},
\end{equation} 
where $\phi_{((k),(m-k))}$ (resp.\ $\psi_{((k),(m-k))}$) is the direct sum of the maps induced by $\phi_\ba$ (resp.\ $\psi_\ba$) and $\gamma_{((k),(m-k))}$ is induced by the inclusion $\mathcal S^{((k),(m-k))}\hookrightarrow\left(\mathbb S^2\right)^m$. 

We have a chain of equalities
\[
\gamma_{((k),(m-k))}\left(\phi_{((k),(m-k))}(M)\right)=\psi_{((k),(m-k))}(M)=L_M,
\]
where the first equality follows from the commutative diagram (\ref{eq:key_comm_diagram}) and the second one follows from Lemma~\ref{lem:image_of_homology_gen}. Hence, by Lemma~\ref{lem:linear_independence}, we see that the elements $$\gamma_{((k),(m-k))}\left(\phi_{((k),(m-k))}(M)\right)$$ are linearly independent if $M$ varies over all the whole set $s\widetilde{\mathbb B}^{((k),(m-k))}$. In particular, it follows that $\phi_{((k),(m-k))}(M)$ must be $\mathbb Z$-linearly independent for $M\in s\widetilde{\mathbb B}^{((k),(m-k))}$. Since the cardinality of $s\widetilde{\mathbb B}^{((k),(m-k))}$ equals the dimension of $H_*(\mathcal S^{((k),(m-k))})$ by Corollary~\ref{cor:Bett_numbers} and Remark~\ref{rem:standard_enriched_cup_diagram_remark}, these elements form a basis of $H_*(\mathcal S^{((k),(m-k))})$. In particular, it directly follows that the map $\gamma_{((k),(m-k))}$ is injective. 
\end{proof}

\begin{thm}\label{thm:cohomology_ring}
The following statements hold:
\begin{enumerate}[(a)]
\item \label{thm:cohomology_part_a} There exists a homotopy equivalence 
\begin{equation}\label{eq:homotopy_equivalence}
\mathcal Fl^{((k),(m-k))}\simeq\mathrm{Sk}_{2(m-k)}^m=\bigcup_{\begin{subarray}{l}
U\subseteq\{1,\ldots,m\}\\ \vert U\vert=m-k\end{subarray}} C_{l_U}\subseteq (\mathbb S^2)^m
\end{equation}
between the $((k),(m-k))$-exotic Springer fiber and the $2(m-k)$-skeleton $\mathrm{Sk}_{2(m-k)}^m$ of the CW-complex $(\mathbb S^2)^m$. 

\item \label{thm:cohomology_part_b} We have an isomorphism of graded algebras 
\begin{equation}\label{eq:cohomology_iso_alg_spr_skel}
H^*(\mathcal Fl^{((k),(m-k))},\mathbb C)\cong H^*(\mathrm{Sk}_{2(m-k)}^m,\mathbb C).
\end{equation}
In particular, using~(\ref{eq:cohomology_iso_alg_spr_skel}), we obtain an explicit presentation of the graded algebra $H^*(\mathcal Fl^{((k),(m-k))},\mathbb C)$ given by 
\begin{equation}\label{eq:cohomology_presentation}
\bigslant{\mathbb C[X_1,\ldots,X_m]}{\left\langle X_i^2, X_I\;
  \begin{array}{|c}
  1 \leq i \leq m,\\
  I\subseteq  \{1,\ldots,m\}, |I|=m-k+1
  \end{array}
  \right\rangle},
\end{equation}
where $X_I=\prod_{i\in I}X_i$ and $\mathrm{deg}(X_i)=2$. Furthermore, the presentation (\ref{eq:cohomology_presentation}) yields a distinguished basis of $H^{2l}(\mathcal Fl^{((k),(m-k))},\mathbb C)$ given by all monomials $X_I$ such that $\vert I\vert=l$, $I\subseteq\{1,\ldots,m\}$, $0\leq l\leq m-k$.
\end{enumerate}
\end{thm}

\begin{proof}
Using the homeomorphism from Theorem~\ref{thm:main_result_1} we can work with $\mathcal S^{((k),(m-k))}$ instead of $\mathcal Fl^{((k),(m-k))}$ when arguing topologically. By the cellular approximation theorem we can replace the natural inclusion of $\mathcal S^{((k),(m-k))}$ into $(\mathbb S^2)^m$ (which in general is not a cellular map) by a homotopic map $j$ which factors through $\mathrm{Sk}_{2(m-k)}^m$ (note that by Remark~\ref{rem:dimension} the complex dimension of $\mathcal Fl^{((k),(m-k))}$ is $m-k$ so there are no cells of real dimension greater than $2(m-k)$) and we obtain a commutative diagram
\begin{equation}\label{eq:comm_diag_cell_approx}
\begin{tikzcd} 
\mathcal S^{((k),(m-k))} \arrow{rr}{j} \arrow{dr}[swap]{j} && \left(\mathbb S^2\right)^m \\
 & \mathrm{Sk}_{2(m-k)}^m \arrow[hookrightarrow]{ru} &
\end{tikzcd}
\end{equation}
The horizontal map $j$ in (\ref{eq:comm_diag_cell_approx}) induces isomorphisms $H_{2l}(\mathcal S^{((k),(m-k))})\cong H_{2l}((\mathbb S^2)^m)$ for all $0\leq l\leq m-k$ because the natural inclusion induces such isomorphisms by Proposition~\ref{prop:injection}. Since the natural inclusion of $\mathrm{Sk}_{2(m-k)}^m$ into $(\mathbb S^2)^m$ is well known to induce isomorphisms $H_{2l}(\mathrm{Sk}_{2(m-k)}^m)\cong H_{2l}((\mathbb S^2)^m)$ for all $0\leq l\leq m-k$, we deduce that $j$ induces isomorphisms $H_{2l}(\mathcal S^{((k),(m-k))})\cong H_{2l}(\mathrm{Sk}_{2(m-k)}^m)$ for all $0\leq l\leq m-k$. Since $\mathcal S^{((k),(m-k))}$ and $\mathrm{Sk}_{2(m-k)}^m$ are simply-connected CW-complexes we can apply the homological version of the Whitehead theorem to the map $j$ to prove that it yields the desired homotopy equivalence $\mathcal Fl^{((k),(m-k))}\cong\mathcal S^{((k),(m-k))}\simeq \mathrm{Sk}_{2(m-k)}^m$. 

Since $\mathcal S^{((k),(m-k))}\simeq \mathrm{Sk}_{2(m-k)}^m$ are homotopy equivalent it suffices to compute the cohomology ring of $\mathrm{Sk}_{2(m-k)}^m$. Note that the map $H^*((\mathbb S^2)^m)\twoheadrightarrow H^*(\mathrm{Sk}_{2(m-k)}^m)$ induced by the inclusion $\mathrm{Sk}_{2(m-k)}^m \hookrightarrow (\mathbb S^2)^m$ is surjective. This follows from the standard isomorphisms $H^{2l}(\mathrm{Sk}_{2(m-k)}^m)\cong H^{2l}((\mathbb S^2)^m)$ for $0\leq l\leq m-k$ and it is obviously true for $l>m-k$ because $H^{2l}(\mathrm{Sk}_{2(m-k)}^m)\cong H_{2l}(\mathrm{Sk}_{2(m-k)}^m)\cong\{0\}$. In particular, the kernel of the surjection is given by the ideal $\bigoplus_{i>2(m-k)}^\infty H^i((\mathbb S^2)^m)$ and we obtain an isomorphism of algebras
\begin{equation}\label{eq:surjection_cohomology}
\bigslant{\displaystyle \bigoplus_{i=0}^\infty H^i((\mathbb S^2)^m)}{ \displaystyle \bigoplus_{i>2(m-k)}^\infty H^i((\mathbb S^2)^m)}\xrightarrow\cong H^*(\mathrm{Sk}_{2(m-k)}^m).
\end{equation} 
Using the isomorphism
\[
H^*((\mathbb S^2)^m)\cong \bigslant{\mathbb Z[X_1,\ldots,X_m]}{\langle X_i^2 \;| 1\leq i\leq m \rangle} 
\]
we see that the quotient in (\ref{eq:surjection_cohomology}) can be identified with the quotient in Theorem~\ref{thm:cohomology_ring} (after additionally switching to complex coefficients).

The statement about the basis of the algebra in (\ref{eq:cohomology_presentation}) is evident.
\end{proof}

\begin{ex}\label{ex:homeo_vs_homotopy}
Note that $\mathrm{Sk}_2^m$ is a bouquet of $m$ two-spheres, i.e., $m$ two-spheres glued together in a single point. It is a standard exercise to show that a bouquet of $m$ two-spheres is homotopy equivalent to a chain of $m$ two-spheres. In particular, for the special case $\mathcal Fl^{((m-1),(1))}\cong\mathcal S^{((m-1),(1))}\simeq\mathrm{Sk}_2^m$ we already obtain the homotopy equivalence (\ref{eq:homotopy_equivalence}) from the discussion in Example~\ref{ex:topological_Springer_fiber}. Note, that $\mathrm{Sk}_{2(m-k)}^m$ and $\mathcal Fl^{((k),(m-k))}$ are in general not homeomorphic, e.g., removing the gluing point of the bouquet $\mathrm{Sk}_2^m$ yields $m$ distinct connected components, whereas removing any point in $\mathcal Fl^{((m-1),(1))}\cong\mathcal S^{((m-1),(1))}$ yields at most two connected components.  
\end{ex}

\subsection{Weyl group action on cohomology} Finally, we define an action of the Weyl group $\mathcal W_{C_m}$ of type $C_m$ on the space $\mathrm{Sk}_{2(m-k)}^m$, describe the induced action in cohomology and compare the resulting representations to Kato's original exotic Springer representation (see Proposition~\ref{prop:Weyl_group_action_cohomology} and Remark~\ref{rem:connection_to_Kato}). 

Recall that the Weyl group $\mathcal W_{C_m}$ of type $C_m$ can be realized as a Coxeter group with generators $s_0,s_1,\ldots,s_{m-1}$ subject to the relations $s_i^2=e$ for all $i$, and for $i,j\neq 0$ $s_is_j=s_js_i$ if $\vert i-j\vert>1$ and $s_is_js_i=s_js_is_j$ if $\vert i-j\vert=1$, and additionally $s_0s_1s_0s_1=s_1s_0s_1s_0$ and $s_0s_j=s_js_0$ for $j\neq 1$. The following proposition is a classical result, see e.g.\ \cite[\S 8.2]{Ser77}, \cite[Theorem 10.1.2]{C93}, \cite[Appendix B]{Mac95} and also~\cite[Proposition 3]{MS16} for a categorical approach. 

\begin{prop}
There exists a bijection between complex, finite-dimensional, irreducible $\mathcal W_{C_m}$-modules (up to isomorphism) and bipartitions of $m$. Given a bipartition $(\lambda,\mu)$ of $m$, we write $V_{(\lambda,\mu)}$ to denote the corresponding irreducible $\mathcal W_{C_m}$-module.
\end{prop}

\begin{rem}
Here, we follow the labeling conventions used in~\cite[\S4]{SW16} (see also Remark~\ref{rem:connection_to_Kato}). In particular, the trivial one-dimensional $\mathcal W_{C_m}$-module is labeled by $((m),(-))$.
\end{rem}

Let $\sigma_0\colon\mathbb R^3\to\mathbb R^3$ be the homeomorphism given by $(x,y,z)\mapsto (-x,y,z)$. It induces a homeomorphism $\mathbb S^2 \xrightarrow\cong\mathbb S^2$ which (by abuse of notation) we also denote by $\sigma_0$.

\begin{prop}\label{prop:Weyl_group_action_cohomology}
Let $0\leq l\leq m-k$. The following statements hold:
\begin{enumerate}[(a)]
\item\label{item:action_on_space} The Weyl group $\mathcal W_{C_m}$ acts on $\mathrm{Sk}_{2(m-k)}^m$ by 
\[
s_0.\left(x_1,x_2,\ldots,x_m\right)=\left(\sigma_0(x_1),x_2,\ldots,x_m\right)\,,
\]
\[
s_i.(x_1,\ldots,x_i,x_{i+1},\ldots,x_m)=(x_1,\ldots,x_{i+1},x_i,\ldots,x_m)\,,\,i\neq 0.
\]
\item\label{item:action_on_cohomology} The induced $\mathcal W_{C_m}$-action on $H^*(\mathrm{Sk}_{2(m-k)}^m,\mathbb C)$ can be described explicitly on the monomial basis of $H^{2l}(\mathrm{Sk}_{2(m-k)}^m,\mathbb C)$ from Theorem~\ref{thm:cohomology_ring} as follows: a generator $s_i$, $i\neq 0$, acts by exchanging $X_i$ and $X_{i+1}$ and leaves $X_j$ invariant for $j\neq i,i+1$, and $s_0$ acts by sending $X_1$ to $-X_1$ and leaves all other $X_i$, $i\neq 1$, invariant.
\item\label{item:identifying_reps} The action from \ref{item:action_on_cohomology} yields an isomorphism $H^{2l}(\mathrm{Sk}_{2(m-k)}^m,\mathbb C)\cong V_{((m-l),(l))}$ of $\mathcal W_{C_m}$-modules.
\end{enumerate}
\end{prop} 
\begin{proof}
Since $\sigma_0$ is an involution of $(\mathbb S^2)^m$ fixing the north pole $p=(0,0,1)$, we see that $s_0.C_{l_U}=C_{l_U}$ for all $U\subseteq\{1,\ldots,m\}$, $\vert U\vert=m-k$. Moreover, for $i\neq 0$ we have $s_i.C_{l_U}=C_{l_{s_i.U}}$, where $s_i.U\subseteq \{1,\ldots,m\}$ is the set containing all elements obtained by applying the permutation $s_i\colon\{1,\ldots,m\}\to\{1,\ldots,m\}$ switching $i$ and $i+1$ to the elements in $U$. Since $\vert s_i.U\vert=m-k$ for $1\leq i\leq m-1$, we conclude that $s_i.(x_1,\ldots,x_m)\in\mathrm{Sk}_{2(m-k)}^m$ for all $(x_1,\ldots,x_m)\in\mathrm{Sk}_{2(m-k)}^m$ and $0\leq i\leq m-1$. It is straightforward to check that the $s_i$, $0\leq i\leq m-1$, satisfy the defining relations of $\mathcal W_{C_m}$ and we obtain a well-defined $\mathcal W_{C_m}$-action on $\mathrm{Sk}_{2(m-k)}^m$ as claimed in Proposition~\ref{prop:Weyl_group_action_cohomology}\ref{item:action_on_space}.     

Note that the $\mathcal W_{C_m}$-action on $\mathrm{Sk}_{2(m-k)}^m$ in~\ref{item:action_on_space} extends to an action on $(\mathbb S^2)^m$. Since the map $\sigma_0$ is homotopic to the antipodal map it follows from standard algebraic topology that the induced action of $s_0$ on $H^*((\mathbb S^2)^m,\mathbb C)\cong\mathbb C[X_1,\ldots,X_m]/(X_i^2)$ is given by sending $X_1$ to $-X_1$ and leaving all other $X_i$ invariant, $i\neq 1$. Moreover, for $i\neq 0$, the induced action of $s_i$ is given by exchanging $X_i$ and $X_{i+1}$ and leaving $X_j$ invariant for $j\neq i,i+1$. Since the inclusion $\mathrm{Sk}_{2(m-k)}^m\subseteq (\mathbb S^2)^m$ is clearly $\mathcal W_{C_m}$-equivariant, the above shows that the $\mathcal W_{C_m}$-action on $\mathrm{Sk}_{2(m-k)}^m$ defined in~\ref{item:action_on_space} induces the action on $H^*(\mathrm{Sk}_{2(m-k)}^m,\mathbb C)$ as claimed in~\ref{item:action_on_cohomology} via the isomorphism (\ref{eq:surjection_cohomology}). 

In order to identify the irreducible representation in each cohomological degree we first note that the $\mathcal W_{C_m}$-action on $(\mathbb S^2)^m$ obtained by extending the action on $\mathrm{Sk}_{2(m-k)}^m$ from~\ref{item:action_on_space} also induces an action on homology $H_*((\mathbb S^2)^m,\mathbb C)\cong\mathbb C[\mathfrak{L}_m]$ which can be described explicitly using the line diagram basis. A generator $s_i$, $i\neq 0$, acts on a line diagram by permuting the lines $i$ and $i+1$, and $s_0$ acts by sending a line diagram to its additive inverse if the first line is undotted and leaves it invariant otherwise. By~\cite[Proposition 40]{SW16} we have an isomorphism of $\mathcal W_{C_m}$-modules $H_{2l}((\mathbb S^2)^m,\mathbb C)\cong V_{((m-l),(l))}$. From the explicit descriptions of the $\mathcal W_{C_m}$-actions it is evident that the map $H_{2l}((\mathbb S^2)^m,\mathbb C)\cong H^{2l}(\mathrm{Sk}_{2(m-k)}^m,\mathbb C)$ sending $l_U$ to $X_{i_1}\ldots X_{i_l}$, where $U=\{i_1,\ldots,i_l\}$, is an isomorphism of $\mathcal W_{C_m}$-modules and we obtain the isomorphisms as claimed in Proposition~\ref{prop:Weyl_group_action_cohomology}\ref{item:identifying_reps}.
\end{proof}

\begin{rem}\label{rem:connection_to_Kato}
By Theorem~\ref{thm:cohomology_ring} we found a topological space isomorphic to the exotic Springer fiber $\mathcal Fl^{((k),(m-k))}$ in the homotopy category of topological spaces. In contrast to the exotic Springer fiber itself (and also its homeomorphic topological model from Theorem~\ref{thm:main_result_1}), this homotopy equivalent space has the property that it admits an action of the Weyl group by Proposition~\ref{prop:Weyl_group_action_cohomology}\ref{item:action_on_space}. By combining Proposition~\ref{prop:Weyl_group_action_cohomology}\ref{item:identifying_reps} with results from \cite{Kat17}, the induced action in cohomology indeed recovers the original exotic Springer representation~\cite{Kat11} and hence can be used as a toy model for Kato's more involved construction. Moreover, by the proof of Proposition~\ref{prop:Weyl_group_action_cohomology} and the discussion in~\cite[\S4]{SW16}, we can identify the monomial basis of our presentation of $H_{2l}(\mathcal Fl^{((k),(m-k))},\mathbb C)\cong H^{2l}(\mathrm{Sk}_{2(m-k)}^m,\mathbb C)$ from Theorem~\ref{thm:cohomology_ring} as a particularly nice cohomology basis. More precisely, under the isomorphism of Proposition~\ref{prop:Weyl_group_action_cohomology}, each monomial corresponds to a basis vector labeled by a standard bitableau if we construct $V_{((m-k),(k))}$ as a Specht module by following, e.g.,~\cite{Can96}.
\end{rem}

\bibliographystyle{amsalpha}
\bibliography{litlist}

\end{document}